\documentclass{amsart}
\usepackage{amsmath,amsthm,amssymb,amsfonts}

\usepackage{graphicx}
\usepackage{tikz}

  \usepackage{epsfig}
\usepackage[pagebackref, colorlinks=true]{hyperref}

\hypersetup{urlcolor=blue, citecolor=blue}

\def\doi#1{   {\href{http://dx.doi.org/#1}
   {{\mdseries\ttfamily DOI}}}}

\usepackage{graphics}

\def\b{\beta}

\newcommand{\al}{\alpha}    \newcommand{\be}{\beta}
\newcommand{\de}{\delta}    \newcommand{\De}{\Delta}
  \newcommand{\ep}{\varepsilon}
  
    \newcommand{\la}{\lambda}

\newcommand{\ga}{\gamma}    
\newcommand{\R}{\mathbb{R}}

\newcommand{\beeq}{\begin{equation}}\newcommand{\eneq}{\end{equation}}

\newcommand{\cmt}[1]{}

\newenvironment{prf}{\noindent {\bf Proof.} }{\endprf\par}
\def \endprf{\hfill  {\vrule height6pt width6pt depth0pt}\medskip}

\def\<{\langle}             \def\>{\rangle}
\def\({\left(}                 \def\){\right)}
\numberwithin{equation}{section}
\newtheorem{thm}{Theorem}[section]
 \newtheorem{cor}[thm]{Corollary}
 \newtheorem{lem}[thm]{Lemma}
 \newtheorem{prop}[thm]{Proposition}
 
 \newtheorem{rem}[thm]{Remark}

\title[spherical maximal operators on radial functions]
{$L^p$-improving bounds for spherical maximal operators over restricted dilation sets: radial improvement}

\author{Shuijiang Zhao$^{*}$}\thanks{* Corresponding author}
\address{School of Mathematical Sciences\\ Zhejiang University\\Hangzhou 310058, P. R. China}\email{zhaoshuijiang@zju.edu.cn }



\begin{document}
\bibliographystyle{plain}

\maketitle
\begin{abstract}
In this paper, we study the spherical maximal operator $ M_E $ over $ E\subset [1,2]$, restricted to radial functions. In higher dimensions $ d\geq 3$, we establish a complete range of $ L^p-$improving estimates for $ M_E $. In two dimensions, sharp results are also obtained for quasi-Assouad regular sets $E$. A notable feature is that the high-dimensional results depend solely on the upper Minkowski dimension, while the two-dimensional results also involve other concepts in fractal geometry such as the Assouad spectrum. Additionally, the geometric shapes of the regions corresponding to the sharp $ L^p-$improving bounds differ significantly between the two cases.
\end{abstract}

\section{Introduction}
For a locally integrable function $ f $, we consider the spherical means
\[A_t f(x) = \int_{\mathbb{S}^{d-1}} f(x-ty) d\sigma(y), \]
where $d\sigma$ denotes the standard normalized surface measure on the unit sphere $\mathbb{S}^{d-1}$ with $ d\geq 2$. For a fixed $ E\subset (0,\infty) $, we define the maximal function
\[ M_E f(x) = \sup_{t\in E} |A_t f(x)|,\]
whose $ L^p $ boundedness and $ L^p-$improving properties have been well studied. In this paper, our purpose is to seek some {\em radial improvement}  for $ L^p-$improving properties of  $ M_E $ when the function $ f $ is restricted to being radial.

Let us first briefly review the history about  $ L^p $ boundedness properties of $ M_E $. When $E=(0,\infty)$, $ M_E $ is bounded on $ L^p$ if and only if $ p>\frac{d}{d-1} $, which was first proven by Stein \cite{1976Stein} for the case $ d\geq 3 $ and later by Bourgain \cite{1986Bourgain} for the case $ d=2$. There is no radial improvement for  the $ L^p $ bounds of $M_{(0,\infty)}$ since the counterexample constructed by Stein \cite{1976Stein} is also radial. At the endpoint $ p=\frac{d}{d-1}$ with $ d\geq 3 $, Bourgain \cite{1985Bourgain} showed a restricted weak type estimate for $M_{(0,\infty)}$ , i.e. $M_{(0,\infty)}$ maps $ L^{\frac{d}{d-1},1}(\R^d)$ to $L^{\frac{d}{d-1},\infty}(\R^d)$. Here and in what follows we use  $ L^{p,q} $ to denote the usual Lorentz space. In two dimensions, this endpoint restricted weak type inequality fails for general functions by a Besicovitch set construction \cite[Proposition 1.13]{2003STW} even though it holds true for radial functions \cite{1987Leckband}. 

For a generic set $E\subset(0,\infty)$, Seeger--Wainger--Wright \cite{1995SWW} used a generalized notion of upper Minkowski dimension to determine the critical exponent $ p(E)$ such that $ M_E $ is bounded on $ L^p $ if $ p>p(E) $ and unbounded if $p<p(E)$. To describe the exponent $ p(E)$, we shall recall some definitions. For a bounded set $F\subset(0,\infty)$, we define $ N(F,a) $ to be the minimal number (which is finite due to the boundedness of $F$) of intervals of length $a$ needed to cover $ E $. Set  $I_k=[2^k, 2^{k+1}]$ for $ k\in\mathbb{Z}$ and $E^k=E\cap I_k$. Then 
\[p(E)=1+\frac{1}{d-1}\overline{\lim_{\de\to 0}}\Big[\sup_{k\in \mathbb{Z}}\frac{\log N(E^k,2^k\de)}{\log\de^{-1}}\Big].\]
Note that $ 1\leq p(E)\leq \frac{d}{d-1} $. Unfortunately, we are unable to achieve any radial improvement on this critical exponent $p(E)$ because of the radial counterexample in \cite[Section 3]{1995SWW}. When $ p=p(E)$, the necessary and sufficient conditions for $ L^p_{\text{rad}}\to L^{p,q}$ estimates with $ p\leq q\leq \infty $ have been established in \cite{1997SWW} except for the extreme case $ p=2$ when $ d=2$. These conditions can be expressed in terms of various summation involving $ N(E^k,2^{k-n})$ with $ k\in \mathbb{Z} $ and $n\in\mathbb{N}$. If some appropriate regularity assumptions  on $E$ are satisfied,  Seeger--Tao--Wright \cite{2003STW} proved that the previous relevant conditions remain necessary and sufficient  for the endpoint $ L^p $ boundedness of $ M_E $ on general functions with $ 1<p\leq \frac{d}{d-1} $ and $ d\geq 2 $. 

As we have seen, it is  hard to achieve significant radial improvement on the $ L^p $ boundedness of $ M_E $ with an unbounded set $ E\subset (0,\infty)$, except for some certain endpoint cases. Now we turn to the $ L^p-$improving  bounds, i.e. $ L^p\to L^q$ estimates for $ p<q$, for $ M_E $ when $ E\subset[1,2]$. In order to review the relevant results, we recall some definitions including various fractal dimensions of $ E$. The {\em (upper) Minkowski dimension} $\dim_{\mathrm{M}}\!E$ is defined by
\[ \dim_{\mathrm{M}}\!E = \inf\big\{ a>0\,:\, \exists\,c>0\;\text{s.t.}\,\forall\,\delta\in(0,1),\, N(E,\delta)\le c\, \delta^{-a}\ \big\}. \]
And the {\em Assouad dimension} $\dim_{\mathrm{A}}\!E$ (see \cite{1979Assouad} for more details) is defined by
\begin{equation}\label{def-assouad}
\dim_{\mathrm{A}}\!E =\inf\left\{ a>0\,:\, \right. \exists\,c>0\; \text{s.t.}\,\forall\,I,\,\delta\in (0,|I|),\,\left. N(E\cap I,\delta)\le c\, (\de/|I|)^{-a}\ \right\};
\end{equation}
here $I$ runs over all intervals. Clearly, $ 0\leq \dim_{\mathrm{M}}\!E\leq  \dim_{\mathrm{A}}\!E\leq 1$. There is another fractal dimension that lies between these two. The definition involves  the
\emph{upper Assouad spectrum} (see \cite{2019FHTTY}) $\overline\dim_{\mathrm{A},\theta}E$ for $ \theta\in [0,1] $, which is defined by 
\[\overline \dim_{\mathrm{A},\theta}E = \inf\left\{ a>0\,:\, \right. \exists\,c>0\;  \text{s.t.}\,\forall\,\delta\in (0,1),\,|I|\ge\delta^\theta,\,
\left. N(E\cap I,\delta)\le c\, (\de/|I|)^{-a}\ \right\}.\]
Then the \emph{quasi-Assouad dimension}  is defined by  the limit
\begin{equation}
\label{eqn:quasi-Assouad-def} \dim_{\mathrm{qA}}\!E = \lim_{\theta\nearrow 1} \overline{\dim}_{\mathrm{A},\theta}E .
\end{equation}
since $ \overline{\dim}_{\mathrm{A},\theta}E$ is non-decreasing with respect to $ \theta\in [0,1] $. As in \cite{2023RS}, we use the {\em  type set} $\mathcal{T}_E$  to represent the region of $ (1/p,1/q)\in  [0,1]^2$ such that $ M_E $ is bounded from $L^p$ to $L^q$.
Similarly, the {\em radial type set} is defined by 
\[\mathcal{T}_E^{rad} = \{ (\tfrac1p,\tfrac1q)\in [0,1]^2\,:\, M_E\;\;\text{is bounded}\;\; L_{rad}^p\to L^q\}. \]
For $0\leq \be\leq \ga \leq 1$, let $ \mathcal{P}(\be,\ga)$ denote the closed  quadrangle formed by the vertices
\begin{align*}
	&O=(0,0),\quad \quad  P_1(\be)=(\tfrac{d-1}{d-1+\be},\tfrac{d-1}{d-1+\be}),\\
	&P_2(\be)=(\tfrac{d-\be}{d-\be+1},\tfrac{1}{d-\be+1}), \quad P_3(\ga)=(\tfrac{d(d-1)}{d^2+2\ga-1},\tfrac{d-1}{d^2+2\ga-1}).
\end{align*}

When $ E=[1,2]$, $ \dim_{\mathrm{M}}\!E=\dim_{\mathrm{qA}}\!E=\dim_{\mathrm{A}}\!E=1$. In two dimensions,  Schlag \cite{1997Schlag} showed that 
\begin{equation}\label{relation-interior}
	\text{int} (\mathcal{P}(1,1))\cup[O,P_1(1))\subset \mathcal{T}_{[1,2]},
\end{equation}
and $ L^p\to L^q$ estimate fails if $ (1/p,1/q)\notin \mathcal{P}(1,1)\setminus \{P_1(1)\}$. Here and in what follows we denote the closed line segment with endpoints $ P $ and $ Q $ by $ [P, Q] $. Similarly, we use the notation $ \left[P, Q\right),  \left(P, Q\right) $ and $  \left(P, Q\right] $. Note that $ \mathcal{P}(1,1) $ becomes a triangle when $ d=2 $ since the points $ P_1(1) $ and $ P_2(1) $ coincide. Lee \cite{2003Lee} obtained the additional endpoint results 
\begin{equation}\label{relation-endpoint}
(O, P_3(1))\cup (P_3(1),P_2(1))\subset \mathcal{T}_{[1,2]},
\end{equation}
as well as the restricted weak type estimate at the point $ P_3(1) $. It is still open whether $ P_3(1)\in \mathcal{T}_{[1,2]}$ for the circular maximal operator. In higher dimensions $ d\geq 3 $, a necessary condition is $\mathcal{T}_{[1,2]}\subset \mathcal{P}(1,1)\setminus[P_1(1),P_2(1)]$, see \cite{1997SS}. When $ d\geq 3 $, the corresponding interior results (\ref{relation-interior}) and  borderline results (\ref{relation-endpoint})   were established successively by Schlag--Sogge \cite{1997SS} and Lee \cite{2003Lee}, respectively. Therefore, the only unknown point for $ \mathcal{T}_{[1,2]} $ when $ d\geq 3 $ is  $ P_3(1)$, at which Lee \cite{2003Lee} obtained the restricted weak type inequality. 

Note that the necessary condition corresponding to the line connecting $ P_3(1) $ and $ P_2(1) $ follows from the Knapp example (see \cite{1997Schlag, 1997SS}), which is non-radial. Consequently, it is reasonable to expect that $\mathcal{T}_{[1,2]}^{rad} $ is significantly larger than $ \mathcal{T}_{[1,2]}$.  However, to the best of our knowledge, there are currently no results regarding  $\mathcal{T}_{[1,2]}^{rad} $, not to mention  $\mathcal{T}_{E}^{rad} $ for a generic $ E\subset [1,2] $. Thus, in our paper, we first consider the $ L^p-$improving bounds for the standard localized spherical maximal operator $ M_{[1,2]}$ acting on radial functions. 

To simplify the statement of our results, we introduce some notation. For $ 0\leq \be\leq 1$, let $ \De(\be) $ denote the closed triangle formed by vertices
\begin{equation}\label{tridef}
O = (0,0),\;	Q_1(\be) = (\tfrac{d-1}{d-1+\beta}, \tfrac{d-1}{d-1+\beta}),\; Q_2(\be) = (\tfrac{d(d-1)}{d^2-1+\be}, \tfrac{d-1}{d^2-1+\be}).
\end{equation}
Our first main result is concerning about $ \mathcal{T}_{[1,2]}^{rad} $. See Figure \ref{fig-interval-higher} and \ref{fig-interval-2D} for the comparison between  $ \mathcal{T}_{[1,2]} $ and $ \mathcal{T}_{[1,2]}^{rad}$.
\begin{thm}\label{thminterval}
	Let $ d\geq 2 $. Then $ \mathcal{T}_{[1,2]}^{rad}=\De(1)\setminus [Q_1(1),Q_2(1)] $.
\end{thm}

	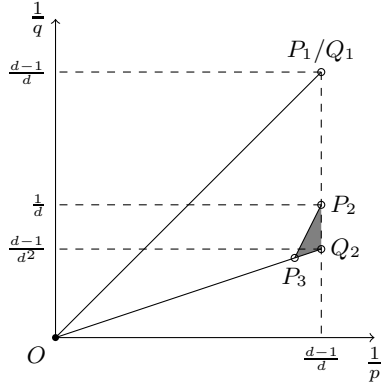
\begin{figure}[h]
	\begin{tikzpicture}[scale=5.3]
	\draw (0,0) [->] -- (0,0.8) node [left] {$\frac1q$};
	\draw (0,0) [->] -- (0.8,0) node [below] {$\frac1p$};
	
	\coordinate (O) at (0,0);
	\coordinate (P1) at (2/3, 2/3);
	\coordinate (P2) at (2/3, 1/3);
	\coordinate (P3) at (3/5, 1/5);	
	\coordinate (Q2) at (2/3, 2/9);
	
	\fill[gray] (P3) -- (Q2) -- (P2)  -- cycle;	
	
	\fill (O)  circle [radius=.25pt];
	\draw (P1)  circle [radius=.25pt];
	\draw (P2)  circle [radius=.25pt];
	\draw (Q2)  circle [radius=.25pt];
	\draw (P3)  circle [radius=.25pt];
	
	\draw (O) -- (Q2);
	\draw (O) -- (P1);
	\draw (P3)  -- (P2);
	\draw [dashed,opacity=.3]  (P1)-- (2/3,0);
	\draw [dashed,opacity=.3]  (P1)-- (0,2/3);
	\draw [dashed,opacity=.3]  (P2)-- (0,1/3);
	\draw [dashed,opacity=.3]  (Q2)-- (0,2/9);

	\node[below left,font=\small] at (O) {$O$};
	\node[above,font=\small] at (P1) {$P_1/Q_1$};
	\node[right,font=\small] at (P2) {$P_2$};
	\node[right,font=\small] at (Q2) {$Q_2$};
	\node[below ,font=\small] at (P3) {$P_3$};
	\node[below ,font=\tiny] at (2/3,0) {$\tfrac{d-1}{d}$};
	\node[left][font=\tiny] at (0,2/3) {$\tfrac{d-1}{d}$};
	\node[left,font=\tiny]at (0,1/3) {$\tfrac{1}{d}$};
	\node[font=\tiny, left] at (0,2/9) {$\tfrac{d-1}{d^2}$};	
	
	\end{tikzpicture}
	\caption{The radial improvement for  $M_{[1,2]}$ in higher dimensions ($ d\geq 3$) is represented by the gray triangle. For better demonstration, we choose the case $ d=3 $ here.}\label{fig-interval-higher}
\end{figure}

\begin{figure}[h]
	\begin{tikzpicture}[scale=7]
	\draw (0,0) [->] -- (0,0.6) node [left] {$\frac1q$};
	\draw (0,0) [->] -- (0.6,0) node [below] {$\frac1p$};
	
	\coordinate (O) at (0,0);
	\coordinate (P1) at (0.5, 0.5);
	\coordinate (Q2) at (.5, .25);
	\coordinate (P3) at (0.4, .2);	
	
	\fill[gray] (P3) -- (Q2) -- (P1)  -- cycle;
	
	\fill (O)  circle [radius=.25pt];
	\draw (P1)  circle [radius=.25pt];
	\draw (Q2)  circle [radius=.25pt];
	\draw (P3)  circle [radius=.25pt];
	
	\draw (O) -- (P1) -- (P3)  -- cycle;
	\draw (P3)  -- (Q2);
	\draw [dashed,opacity=.3]  (P1)-- (0.5,0);
	\draw [dashed,opacity=.3]  (P1)-- (0,0.5);
	\draw [dashed,opacity=.3]  (Q2)-- (0,0.25);
	\draw [dashed,opacity=.3]  (P3)-- (0,0.2);
	\draw [dashed,opacity=.3]  (P3)-- (0.4,0);

	\node[below left,font=\small] at (O) {$O$};
	\node[above,font=\small] at (P1) {$P_1/P_2/Q_1$};
	\node[right,font=\small] at (Q2) {$Q_2$};
	\node[below,font=\small] at (P3) {$P_3$};
	\node[below,font=\tiny] at (0.5,0) {$1/2$};
	\node[below,font=\tiny] at (0.4,0) {$2/5$};
	\node[left][font=\tiny] at (0,0.5) {$1/2$};
	\node[left,font=\tiny]at (0,0.25) {$1/4$};
	\node[font=\tiny, left] at (0,0.2) {$1/5$};
	\end{tikzpicture}	
	\caption{The gray triangle  illustrates the radial improvement for the circular maximal operators $ M_{[1,2]} $.}\label{fig-interval-2D}
\end{figure}
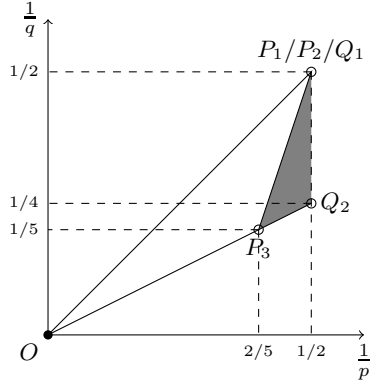

In higher dimensions $ d\geq 3 $, $ \mathcal{T}_E $ has also been extensively studied. For a quasi-Assouad regular set $ E \subset[1,2]$ with $\be= \dim_{\mathrm{M}}\!E$ and $\ga=\dim_{\mathrm{qA}}\!E$, a necessary condition $\mathcal{T}_E \subset \mathcal{P}(\be,\ga)$ was obtained in \cite{2021AHRS,2023RS}. Here, a set $ E\subset[1,2] $ is called \emph{quasi-Assouad regular} if either $ \ga=0 $ or $\overline\dim_{\mathrm{A},\theta}E=\ga$ for any $ \theta\in(1-\be/\ga,1)$. For the endpoint estimates on the off-diagonal boundaries of $\mathcal{P}(\be,\ga)$, we shall recall some notation in \cite[Section 2.3]{2023RS}. We define the $\beta$-Minkowski characteristic  $\chi^E_{\mathrm{M},\beta}:(0,1]\to [0,\infty]$ with 
\begin{equation}\label{eqn:mink-char}
\chi^E_{\mathrm{M},\beta}(\delta) = \delta^\beta N(E,\delta),
\end{equation}
and  $\gamma$-Assouad characteristic $\chi^E_{\mathrm{A},\gamma}:(0,1]\to [0,\infty]$ with
\begin{equation}\label{eqn:assouad-char}
\chi^E_{\mathrm{A},\gamma}(\delta) = \sup_{\substack{ |I|\ge \delta}}\big(\tfrac{\delta}{|I|}\big)^\gamma N(E\cap I,\delta).
\end{equation}
When $ \sup_{0<\de<1}\chi^E_{\mathrm{M},\beta}(\delta)<\infty $  with $ 0\leq \be<1 $ and $ \sup_{0<\de<1}\chi^E_{\mathrm{A},\ga}(\delta)<\infty$ with $ 0\leq \ga\leq 1$, Anderson--Hughes--Roos--Seeger \cite{2021AHRS} obtained the restricted weak type estimates at the points $ P_1(\be), P_2(\be)$ and $ P_3(\ga)$, and then proved  the endpoint result
\begin{equation}\label{endpoint-genericE}
	 \mathcal{P}(\be,\ga)\setminus\{P_1(\be),P_2(\be),P_3(\ga)\}\subset \mathcal{T}_E 
\end{equation}
by interpolation. For a generic $ E\subset[1,2]$ with $\be= \dim_{\mathrm{M}}\!E, \ga=\dim_{\mathrm{qA}}\!E$ and $\ga_*=\dim_{\mathrm{A}}\!E$, the non-endpoint result
\begin{equation}\label{nonendpoint-genericE}
	\text{int}(\mathcal{P}(\be,\ga))\cup(O,P_1(\be))\cup(O,P_3(\ga_*))\subset \mathcal{T}_E
\end{equation}
was also essentially established in \cite{2021AHRS}, which is sharp at least for quasi-Assouad regular sets. Indeed, (\ref{nonendpoint-genericE}) follows directly from the endpoint result (\ref{endpoint-genericE}) by a standard argument (see also the proof of Corollary \ref{cor2D-R3}).

Our second main result  gives a complete characterization of $\mathcal{T}_E^{rad}$ in higher dimensions $ d\geq 3 $, which depends solely on the properties of the covering number $ N(E,\de) $.

\begin{thm}\label{thmhigherdim}
	Let $ d\geq 3 $ and $ E\subset [1,2] $ with $ \dim_{\mathrm{M}}\!E=\be $.
	
	(i) Suppose that $ \be<1 $ and 
	\begin{equation}\label{minassm}
		\sup_{0<\de<1}\chi^E_{\mathrm{M},\beta}(\delta)<\infty.
	\end{equation}
	Then $ \mathcal{T}_E^{rad} =\De(\be) $.
	
	(ii) Suppose that (\ref{minassm}) does not hold true. Then $ \mathcal{T}_E^{rad} =\De(\be)\setminus [Q_1(\be), Q_2(\be)] $.
	
	(iii) Suppose that $ \tfrac{d}{d-1}\leq q\leq \tfrac{d^2}{d-1} $. Then $ M_E $ is bounded $ L^{\tfrac{d}{d-1}}_{rad}\to L^q $ if and only if 
	\begin{equation}\label{minlogloss}
		\sup_{0<\de<1}(\log(\tfrac{1}{\de}))^{\tfrac{q}{d}}\chi^E_{\mathrm{M},1}(\delta)<\infty.
	\end{equation}
\end{thm}
\begin{rem}
	As before, it is no surprise that $ \mathcal{T}_E^{rad} $ is larger than $\mathcal{T}_E$. When $ d\geq 3 $, a new distinction arises from their geometric shapes. According to \cite[Theorem 1.2]{2023RS}, $ \overline{\mathcal{T}_E}$ can be any closed convex set $ \mathcal{W} $ such that $ \mathcal{P}(\be,\ga)\subset\mathcal{W} \subset \mathcal{P}(\be,\be)$ for some $ 0\leq \be\leq \ga\leq 1 $. However, it follows from Theorem \ref{thminterval} and \ref{thmhigherdim}, $\overline{\mathcal{T}_E^{rad}}$ is always a  triangle  for any $ E\subset[1,2] $ when $ d\geq 3 $.
\end{rem}

\begin{figure}
	\begin{tikzpicture}[scale=5]
	\draw (0,0) [->] -- (0,0.85) node [left] {$\frac1q$};
	\draw (0,0) [->] -- (0.85,0) node [below] {$\frac1p$};
	
	\coordinate (O) at (0,0);
	\coordinate (P1) at (2/2.8, 2/2.8);
	\coordinate (P2) at (2.2/3.2, 1/3.2);
	\coordinate (P3) at (3/5, 1/5);	
	\coordinate (Q2) at (6/8.8, 2/8.8);
	
	\fill[gray] (P3) -- (Q2) -- (P2)  -- cycle;	
	
	\fill (O)  circle [radius=.25pt];
	\fill (P1)  circle [radius=.25pt];
	\fill (P2)  circle [radius=.25pt];
	\fill (Q2)  circle [radius=.25pt];
	\fill (P3)  circle [radius=.25pt];
	
	\draw (O) -- (P1)-- (Q2) -- cycle;
	
	\node[below left,font=\small] at (O) {$O$};
	\node[above,font=\small] at (P1) {$P_1(\beta)/Q_1(\beta)$};
	\node[right,font=\small] at (P2) {$P_2(\beta)$};
	\node[right,font=\small] at (Q2) {$Q_2(\beta)$};
	\node[below ,font=\small] at (P3) {$P_3(\gamma)$};
	\end{tikzpicture}
	\caption{When $d\geq3$, the gray triangle shows the distinction between $ \overline{\mathcal{T}_E}$ and $ \overline{\mathcal{T}_E^{rad}}$ for a quasi-Assouad regular set $E$.}\label{fig-E-higher}
\end{figure}
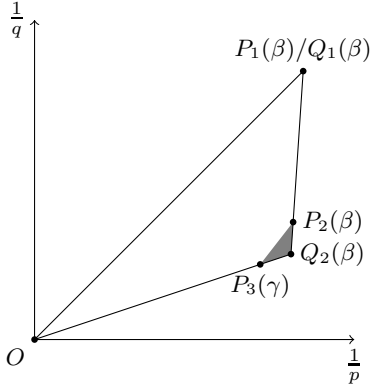

In two dimensions, there is a slight difference for $ \mathcal{T}_E $. The same result (\ref{endpoint-genericE}) holds true in \cite{2021AHRS} under the assumptions $ \sup_{0<\de<1}\chi^E_{\mathrm{M},\beta}(\delta)<\infty $  with $ 0\leq \be<1 $ and $ \sup_{0<\de<1}\chi^E_{\mathrm{A},\ga}(\delta)<\infty$ with $ 0\leq \ga<1/2$. However, it remains open whether the restricted weak type estimate holds at the point $ P_3(\ga) $ when $ 1/2\leq \ga<1$. Once this estimate can be verified, we can easily obtain almost all strong type estimates on the off-diagonal boundaries. Nevertheless, Roos--Seeger \cite{2023RS} got the sharp $ L^p-$improving bounds for any  quasi-Assouad regular set $ E $ with  $\be= \dim_{\mathrm{M}}\!E$ and $\ga=\dim_{\mathrm{qA}}\!E$, in the sense that $ \overline{\mathcal{T}_E}=\mathcal{P}(\be,\ga)$. 

Now we consider $\mathcal{T}_E^{rad} $ in two dimensions. Unlike the high-dimensional case for $\mathcal{T}_E^{rad} $, the relevant properties of $ N(E,\de) $ alone are not sufficient to determine the range of $ L^p-$improving estimates for circular maximal operator $ M_E $ on radial functions. It turn out that $ \mathcal{T}_E^{rad} $ when $ d=2 $ is also closely related to the (quasi-)Assouad dimension of $ E$. Similar to the previous mentioned $ \mathcal{T}_E $ in two dimensions, we also need to conduct a case-by-case analysis. Specifically, the critical case for classification is when $ 2\dim_{\mathrm{qA}}\!E $ equals $\dim_{\mathrm{M}}\!E+1$. Before this,  we introduce some notation. For $ d=2 $ and $ 1\leq \be+1\leq 2\ga\leq 2 $, let
\[Q_3(\be, \ga)=\big(\tfrac{2-\be(1-\theta)}{2(1+\ga \theta)}, \tfrac{1}{2(1+\ga \theta)}\big),\]
where we define $ \theta=\theta(\be, \ga)=\tfrac{1-\be}{2(\ga-\be)}\in [0,1] $ for $ \be<1 $ and $ \theta(1,1)=1$ for $ \be=\ga=1$. Let $ \mathcal{Q}(\be, \ga) $ denote the closed quadrilateral formed by vertices $ O$, $Q_1(\be)$, $Q_3(\be, \ga)$, $Q_2(2\ga-1) $. Note that the quadrilateral $ \mathcal{Q}(\be, \ga)$ degenerates into the triangle $\De(\be)$ when $2\ga=\be+1$.
\begin{thm}\label{thmcircular}
	Let $ d=2 $ and $ E\subset [1,2] $ with $ \dim_{\mathrm{M}}\!E=\be $, $\dim_{\mathrm{qA}}\!E=\ga$, $\dim_{\mathrm{A}}\!E=\ga_*$.
	
	(i) Suppose that $ 2\ga<\be+1 $. Then
	\begin{equation}\label{case1}
		\text{int}(\De(\be))\cup[O, Q_1(\be)\big)\cup \big(O, Q_2(\max\{\be, 2\ga_*-1\})\big)\subset \mathcal{T}_E^{rad}.
	\end{equation}
	
   (ii) Suppose that $ 2\ga\geq\be+1 $. Then
	\begin{equation}\label{case2}
		\text{int}(\mathcal{Q}(\be,\ga))\cup[O, Q_1(\be)\big)\cup \big(O, Q_2(2\ga_*-1)\big)\subset \mathcal{T}_E^{rad}.
	\end{equation}
\end{thm}

Besides the easy inclusion $  \mathcal{T}_E^{rad}\subset \De(\be)$, surprisingly, we find some extra necessary conditions related to (quasi-)Assouad dimension when $ d=2 $. Here we call $ M_E $ is of radial strong type $ (p,q) $ if $ (1/p,1/q) \in \mathcal{T}_E^{rad}$ and of radial restricted weak type $ (p,q) $ if $M_E $ is bounded from $ L^{p,1}_{rad}$ to $ L^{q,\infty}$.
\begin{thm}\label{thmnecessary}
		Let $ d=2 $ and $ E\subset [1,2] $ with $ \dim_{\mathrm{M}}\!E=\be $, $\dim_{\mathrm{qA}}\!E=\ga$, $\dim_{\mathrm{A}}\!E=\ga_*$. If $ M_E $ is of  radial restricted weak type $ (p,2p) $, then $ p\geq \max\{\tfrac{\be+3}{2},\ga_*+1\} $. Additionally,  suppose that $ E $ is quasi-Assouad regular and  $ 2\ga\geq\be+1 $. Then $ \mathcal{T}_E^{rad}\subset \mathcal{Q}(\be,\ga)$.
\end{thm}
\begin{rem}
	Let $ E $ satisfy all the assumptions in Theorem \ref{thmnecessary}. By combining with part (ii) of Theorem \ref{thmcircular}, we obtain that  $\overline{\mathcal{T}_E^{rad}}=\mathcal{Q}(\be,\ga)$. Compared to the high-dimensional case, the radial improvement in two dimensions is somewhat narrowed. See Figure \ref{fig-E-higher} and \ref{fig-E-2D} for more details.
\end{rem}

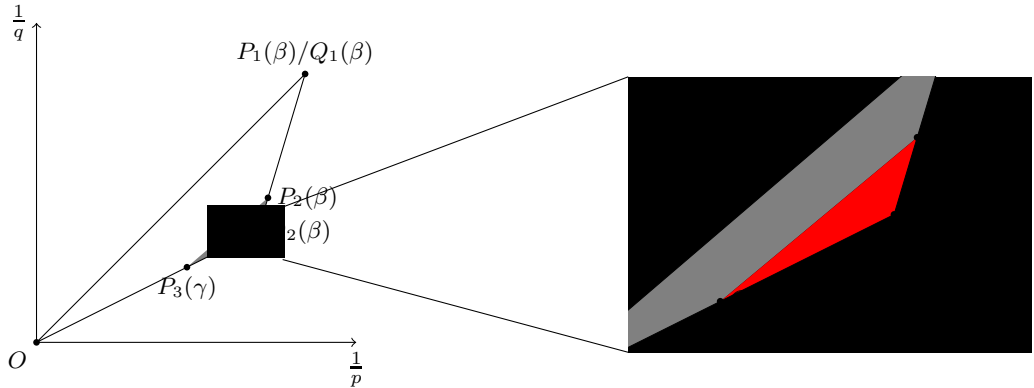
\begin{figure}[ht]
	\begin{tikzpicture}[scale=5]
	\draw (0,0) [->] -- (0,0.85) node [left] {$\frac1q$};
	\draw (0,0) [->] -- (0.85,0) node [below] {$\frac1p$};
	
	\coordinate (O) at (0,0);
	\coordinate (P1) at (1/1.4, 1/1.4);
	\coordinate (P2) at (1.6/2.6,1/2.6);
	\coordinate (P3) at (2/5, 1/5);	
	\coordinate (Q2b) at (2/3.4, 1/3.4);
	\coordinate (Q2g) at (0.5, 0.25);
	\coordinate (Q3) at (1.8/3, 1/3);
	
	\fill[gray] (P3) -- (Q2g) -- (Q3) -- (P2)  -- cycle;
	\fill[red] (Q2g) -- (Q2b) -- (Q3)  -- cycle;
	
	\fill (O)  circle [radius=.25pt];
	\fill (P1)  circle [radius=.25pt];
	\fill (P2)  circle [radius=.25pt];
	\fill (P3)  circle [radius=.25pt];
	\fill (Q2b)  circle [radius=.25pt];
	\fill (Q2g)  circle [radius=.25pt];
	\fill (Q3)  circle [radius=.25pt];
	
	\draw (O) -- (P1)-- (Q3) --(Q2g)-- cycle;
	\draw [dashed,opacity=.3]  (Q2g)-- (Q3)--(Q2b)--cycle;
	
	\node[below left,font=\small] at (O) {$O$};
	\node[above,font=\small] at (P1) {$P_1(\beta)/Q_1(\beta)$};
	\node[right,font=\small] at (P2) {$P_2(\beta)$};
	\node[right,font=\small] at (Q2b) {$Q_2(\beta)$};
	\node[below ,font=\small] at (P3) {$P_3(\gamma)$};
	
	\coordinate (LL) at (2/3.4-0.135, 1/3.4-0.07);
	\coordinate (HH) at (2/3.4+0.07, 1/3.4+0.07);
	\coordinate (HL) at (2/3.4+0.07, 1/3.4-0.07);
	\coordinate (LH) at (2/3.4-0.135, 1/3.4+0.07);
	\fill [opacity=.05] (LL) rectangle (HH);
	
	\begin{scope}[xshift=-0.8cm, yshift=-1.2cm, scale=5.25]
	\coordinate (SO) at (0,0);
	\coordinate (SP1) at (1/1.4, 1/1.4);	
	\coordinate (SP2) at (1.6/2.6,1/2.6);
	\coordinate (SP3) at (2/5, 1/5);		
	\coordinate (SQ2b) at (2/3.4, 1/3.4);
	\coordinate (SQ2g) at (0.5, 0.25);
	\coordinate (SQ3) at (1.8/3, 1/3);
	
	\coordinate (SLL) at (2/3.4-0.135, 1/3.4-0.07);
	\coordinate (SHH) at (2/3.4+0.07, 1/3.4+0.07);
	\coordinate (SHL) at (2/3.4+0.07, 1/3.4-0.07);
	\coordinate (SLH) at (2/3.4-0.135, 1/3.4+0.07);
	\draw[solid,opacity=.1] (HL)--(SLL);
	\draw[solid,opacity=.1](HH)--(SLH);
	\fill [opacity=.02] (SLL) rectangle (SHH);
	\clip (SLL) rectangle (SHH);
	
	\fill[gray] (SP3) -- (SQ2g) -- (SQ3) -- (SP2)  -- cycle;
	\fill[red] (SQ2g) -- (SQ2b) -- (SQ3)  -- cycle;
	
	\draw (SP3)--(SP2);
	\draw (SO)--(SQ2b)--(SP1);

	\fill (SQ2b)  circle [radius=.05pt];
	\fill (SQ2g)  circle [radius=.05pt];
	\fill (SQ3)  circle [radius=.05pt];
	
	\node[right,font=\small] at (SQ2b) {$Q_2(\beta)$};
	\node[right,font=\small] at (SQ3) {$Q_3(\beta,\gamma)$};
	\node[right,font=\small] at (SQ2g) {$Q_2(2\gamma-1)$};
	\end{scope}
	\end{tikzpicture}
	\caption{For a quasi-Assouad regular set $E$ with $ 2\gamma>\beta+1$, the red triangle is excluded from the radial improvement for the circular maximal operator $M_E$, highlighting the difference between two-dimensional and higher-dimensional cases.}\label{fig-E-2D}
\end{figure}

If we consider a much larger class of $ E $, the geometric shape of $ \overline{\mathcal{T}_E^{rad}}$ can be more complicated. For example, we can get the following corollary for finite unions of quasi-Assouad regular sets. It is an analogue of \cite[Theorem 1.4]{2023RS} and can be directly obtained from Theorem \ref{thmcircular} and \ref{thmnecessary}. 
\begin{cor}
	Let $ d=2 $ and $ E=\cup_{j=1}^mE_j$ where $ E_j\subset [1,2] $ is a quasi-Assouad regular set with $\be_j=\dim_{\mathrm{M}}\!E$, $\ga_j=\dim_{\mathrm{qA}}\!E$ and $ 2\ga_j\geq \be_j+1$. Then $\overline{\mathcal{T}_E^{rad}}=\cap_{j=1}^m \mathcal{Q}(\be_j,\ga_j)$.
\end{cor}
Finally, we consider the endpoint estimates on the boundaries of $ \mathcal{Q}(\be,\ga) $ or $\De(\be)$ in two dimensions.
\begin{thm}\label{thm2Dendpoint}
		Let $ d=2 $, $ E\subset [1,2] $, $ 0\leq \be<1 $ and $ \be\leq \ga\leq 1 $. Suppose that $ \sup_{0<\de<1}\chi^E_{\mathrm{M},\beta}(\delta)<\infty $, $\sup_{0<\de<1}\chi^E_{\mathrm{A},\ga}(\delta)<\infty $.
		
		(i) If $2\ga<\be+1 $, then $  \De(\be)\subset \mathcal{T}_E^{rad}$.
		
		(ii) If $2\ga\geq \be+1 $, then $\mathcal{Q}(\be,\ga)\setminus\{Q_2(2\ga-1)\}\subset \mathcal{T}_E^{rad}$. Moreover, $ M_E $ is of radial restricted weak type $ (p,q) $ when $ (\tfrac{1}{p},\tfrac{1}{q}) $ is  the point $ Q_2(2\ga-1) $.
\end{thm}
\begin{rem}
	If logarithmic decay like $ \log^{-\al}(1/\de) $ with some suitable $ \al>0 $ is imposed on  $ \chi^E_{\mathrm{M},1}(\delta) $ and $ \chi^E_{\mathrm{A},1}(\delta)$, using the methods from Section \ref{sect-circular}, we can obtain the endpoint radial strong type  estimate for some point $ Q$ on the vertical line $ [Q_1(1),Q_2(1)]$. However, the exponent  $ \al>0  $ is far away from being sharp.
\end{rem}
\subsection*{Overview of our argument}
We always denote a radial function  $f$ in $ \R^d $ by $ f(x)=f_0(r)$, where $ r=|x| $. A simple but important observation is  that when $ f $ is a radial function, $ A_t f$ is also  radial. Indeed, it follows from a straightforward computation by Leckband \cite[Lemma 1]{1987Leckband} that 
\begin{equation}\label{radialformula}
A_tf(x)=c_d\int_{|r-t|}^{r+t}K_t(r, s)f_0(s)ds,
\end{equation}
where 
\begin{equation}\label{kernel}
K_t(r,s)=\Bigg[\frac{\sqrt{(r+t)^2-s^2}\sqrt{s^2-(r-t)^2}}{(r+t)^2-(r-t)^2}\Bigg]^{d-3}\frac{s}{(r+t)^2-(r-t)^2}.
\end{equation}
Thus, it suffices to study the maximal estimates for this one-dimensional integral. However, unlike the standard Hardy--Littlewood maximal operator,  the  kernel $ K_t(r,s)$ of this integral is somewhat complicated. Fortunately, in higher dimensions $ d\geq 3$,  $ K_t(r,s)$ is essentially free of singularities. By dividing $ r $ into different cases and performing simple calculations, we  reduce the problem to estimating a few new integrals involving only $ f_0 $ and potential weight $ s^\al $ for some $ \al>0 $.  Among these case, $ r\sim t $  is the most  challenging, as the left endpoint $ |r-t| $ of integral interval can approach zero. To address this, we naturally decompose  the range of $r$ based on its distance from $E$. See Section \ref{sect-higher} for further details.

In two dimensions, $ K_t(r,s)$ exhibits more intricate singularities due to the factors $1/\sqrt{r+t-s}$ and $1/\sqrt{s-|r-t|}$. For simplicity, we focus on the singularities arising from $1/\sqrt{s-|r-t|}$. As before, we divide $ r $ into three cases: $ t\leq r/2$,  $t\geq 3r/2$ and  $r/2<t<3r/2$. For the first case, we perform a dyadic decomposition on both the range of $ r $ and the integral interval $ [r-t,r+t]$.  Using interpolation, we  derive the $ L^p-$improving estimates for the decomposed operators. In this context, the $ L^1\to L^1 $ operator norm involve the covering numbers $ N(E,\de) $. We  proceed similarly in the second case, with a key difference: at this point, the original integration interval $ [t-r, t+r] $ has length $2r$, which can be much smaller than $1$. Consequently, when calculating the multiplicities, local covering numbers $ N(E\cap I,\de)$  naturally arise. This motivates the introduction of  the concept of (quasi-)Assouad dimension. For the final case, the integral includes the singular term $1/\sqrt{s-|r-t|}$ and potentially a weight $ s^\al $ for some $ \al>0 $. Here, we  perform a dyadic decomposition on both  $|r-t|$ and the integral interval $ [|r-t|,r+t]$. However, it requires  a more refined decomposition on the range of $ r$ and $ E $  to ensure that $ |r-t| $ is contained within a binary interval. To achieve this, we essentially follow the argument in \cite[Proposition 5.4]{1997SWW}.

\subsubsection*{Outline} Our paper is organized as follows. In Section \ref{sect-higher}, we consider the three cases $ E=[1,2] $, $ \dim_{\mathrm{M}}\!E<1 $ and $ \dim_{\mathrm{M}}\!E=1 $ in  higher dimensions $ d\geq 3$. In Section \ref{sect-circular}, we establish the more complicated results for the circular maximal operator. Various necessary conditions including the proof of Theorem \ref{thmnecessary} are discussed in Section \ref{sect-nece}.

\subsubsection*{Notation}
For $ d\geq 2$, $ \mu_d $ is a measure on $ \R^+ $ defined by $ d\mu_d=r^{d-1}dr $.  For a set $ A\subset \R^d$, we write $ |A| $ for the usual Lebesgue measure of $A$. We define $ I_k=[2^k, 2^{k+1}] $ for $ k\in\mathbb{Z}$, whose double interval is denoted by $ \tilde{I}_k$. $A\lesssim B $ means that there exists a positive irrelevant constant $ C $ such that $ A\leq CB $. Similarly,  we write $ A\ll B $ to mean that $ A$ is far less than $ B $. We use $ A\sim B $ if $ A\lesssim B $ and $ B\lesssim A $. 

\section{$L^p$-improving bounds in higher dimensions}\label{sect-higher}
In this section, we basically follow the argument in \cite{1997SWW} to obtain the $L^p$-improving estimates for $ M_{E} $ when $ d\geq 3 $. Since we are dealing with the localized maximal operator, the proof here is somewhat simplified compared to the previous one.

\subsection{The case when $ E=[1,2] $} \label{subsec localized in higher}
Here we give a proof of Theorem \ref{thminterval} in higher dimensions $ d\geq 3 $. First, we use a pointwise inequality for $ M_E $ acting on radial functions in \cite{1997SWW}, which is an easy consequence of (\ref{radialformula}) and (\ref{kernel}).
\begin{lem}[\cite{1997SWW}, Lemma 3.1]
		Let $ d\geq 3 $, $ E\subset [1,2] $ and $ f $ is a radial function defined in $ \R^d $. Then
		\begin{equation}\label{pointwise-higher}
			M_Ef(x)\lesssim [\mathfrak{M}_{E,p}g(r)+R_1f_0(r)+R_2f_0(r)],
		\end{equation}
		where
		\[ g(s)=f_0(s)s^{\tfrac{d-1}{p}},\]
		\begin{equation}\label{mainpart}
			\mathfrak{M}_{E,p}g(r)=\chi_{(2/3,4)}(r)\sup_{\substack{t \in E \\ r/2<t<3r/2}}\int_{|r-t|}^{r+t}s^{\tfrac{d-1}{p'}-1}|g(s)|ds,
		\end{equation}
		and
		\begin{align*}
			R_1f_0(r)&=\chi_{[2,\infty)}(r)\sup_{\substack{t \in [1,2] \\ t\leq r/2}}\int_{r-t}^{r+t}|f_0(s)|ds,\\
			R_2f_0(r)&=\chi_{[0,4/3)}(r)\sup_{\substack{t \in [1,2] \\ t\geq 3r/2}}\frac{1}{r}\int_{t-r}^{t+r}|f_0(s)|ds.
		\end{align*}
\end{lem}
We have the following $ L^p $-improving estimates for $ R_1 $ and $ R_2 $.
\begin{lem}\label{lemR1}
	For $ 1\leq p\leq q \leq \infty $ and $d\geq 1 $, the operator $ R_1 $ is bounded from $ L^p(\mu_d) $ to $ L^q(\mu_d)$.
\end{lem}
\begin{proof}
	Since $ R_2 $ is bounded on $ L^\infty(\mu_d) $, it suffices to prove that  $ R_1 $ is $ L^1(\mu_d)\to L^1(\mu_d) $ and $ L^1(\mu_d)\to L^{\infty}(\mu_d) $ bounded. Then this lemma is an easy consequence of real interpolation theorem. Note that 
	\begin{align*}
	R_1f_0(r)\lesssim&\chi_{[2,\infty)}(r)r^{1-d}\sup_{\substack{t \in [1,2] }}\int_{r-t}^{r+t}|f_0(s)|s^{d-1}ds\\
	         \lesssim& \chi_{[2,\infty)}(r)r^{1-d}\int_{r-2}^{r+2}|f_0(s)|s^{d-1}ds
	\end{align*}
Then it is easy to see that $ \|R_1f_0\|_{L^\infty(\mu_d)} \lesssim\|f_0\|_{L^1(\mu_d)}$. Moreover, by Fubini's theorem, we have 
\begin{align*}
\int |R_1f_0(r)|r^{d-1}dr\leq &\int_{2}^{\infty}\int_{r-2}^{r+2}|f_0(s)|s^{d-1}dsdr\\
                          =&\int_{0}^\infty \big(\int_{\max\{2,s-2\}}^{s+2}dr\big)|f_0(s)|s^{d-1}ds	\\
                          \lesssim&\|f_0\|_{L^1(\mu_d)}.
\end{align*}
\end{proof}
\begin{lem}\label{lemR2}
	Suppose that $ d>1 $, $ 1<p\leq \infty $ and $ p\leq q\leq pd $. Then the operator $ R_2 $ is bounded from $ L^p(\mu_d) $ to $ L^q(\mu_d)$. Moreover, $ R_2 $ is of strong type $ (1,1) $  and weak type $ (1,d) $, with respect to the measure $ \mu_d $.
\end{lem}
\begin{proof}
	It is trivial that $ R_2 $ is bounded on $ L^\infty(\mu_d) $. Thus, by real interpolation theorem, it suffices to show that $ R_2 $ is $ L^1(\mu_d)\to L^1(\mu_d) $ and $ L^1(\mu_d)\to L^{d,\infty}(\mu_d) $ bounded. Note that 
	\[R_2f_0(r)\leq \frac{1}{r}\|f_o\|_{L^1(\mu_d)}\chi_{[0,4/3)}(r).\]
	Then we easily obtain that $ R_2 $ is bounded from $ L^1(\mu_d)$ to $L^{d,\infty}(\mu_d)  $. Moreover, we have
	\[\int |R_2f_0(r)|r^{d-1}dr\leq \int_{0}^{4/3}r^{d-2}\|f_o\|_{L^1(\mu_d)}dr\lesssim \|f_0\|_{L^1(\mu_d)}.\]
	The last inequality follows from the  assumption $ d> 1 $.
\end{proof}

Now we give a proof for the sufficient part of Theorem \ref{thminterval} in higher dimensions $ d\geq 3 $. Actually, we can obtain some Lorentz estimates for the endpoints $ Q_1(1) $ and $ Q_2(1) $.
\begin{prop}
	Let $ d\geq 3 $ and $ p_d=\frac{d}{d-1} $. Then $ M_{[1,2]} $ is bounded from  $ L_{rad}^{p_d,1}(\R^d) $  to  the Lorentz spaces $L_{rad}^{p_d,1}(\R^d) $ and $L_{rad}^{p_dd,1}(\R^d) $.
\end{prop}
\begin{proof}
	By real interpolation theorem, Lemma \ref{lemR1} and Lemma \ref{lemR2}, both of the operators $ R_1 $ and  $ R_2 $ map $ L^{p_d,1}(\mu_d) $ boundedly to  $L^{p_d,1}(\mu_d) $ and $L^{p_dd,1}(\mu_d) $. We claim that, for any $ p>0 $ and $ s>0 $, 
	\[\|\mathfrak{M}_{[1,2],p_d}g\|_{L^{p,s}(\mu_d)}\lesssim\|f_0\|_{L^{p_d,1}(\mu_d)}.\]
	  Indeed, by generalized H\"{o}lder inequality,  we have
	 \begin{align*}
	 	\mathfrak{M}_{[1,2],p_d}g(r)\leq & \chi_{[2/3,4]}(r)\int_{0}^6 |f_0(s)|\frac{1}{s}d\mu_d(s)\\
	 	                         \leq &\chi_{[2/3,4]}(r)\||\cdot|^{-1}\|_{L^{d,\infty}(\mu_d)} \|f_0\|_{L^{p_d,1}(\mu_d)}\\
	 	                         \lesssim&\chi_{[2/3,4]}(r)\|f_0\|_{L^{p_d,1}(\mu_d)}.
	 \end{align*}
	 Then we finish  the proof in view of (\ref{pointwise-higher}).
\end{proof}
\begin{rem}
	 When $ p_d<q<p_dd $, $ M_{[1,2]} $ is bounded from  $ L_{rad}^{p_d,1}(\R^d) $  to $ L^{q,s}(\R^d) $  for any $ s>0 $ by interpolation between $ Q_1(1) $ and $ Q_2(1) $, and  radial strong type estimates for $ M_{[1,2]} $ follow from interpolation with the trivial $ L^\infty\to L^\infty $ estimate.
\end{rem}

Here we consider the necessary part of Theorem \ref{thminterval} in all dimensions $ d\geq 2 $.  By standard radial counterexamples (see Section \ref{stand1}-\ref{stand3}), it is easy to see that $  \mathcal{T}_{[1,2]}^{rad}\subset\De(1) $. So we only need to exclude the line segment $ [Q_1(1), Q_2(1)] $. Consider  Stein's counterexample \cite{1976Stein}:
 \[F(x)=|x|^{-(d-1)}\log^{-1}\big(\frac{1}{|x|}\big)\chi_{[0,1/2]}(|x|).\]
 Note that $ M_{[1,2]}F(x)=\infty $ if $ |x|\in [1,2] $. However, by simple calculation, we know that $ F\in L_{rad}^{p_d,r}(\R^d) $ for any $ r>1 $. Thus, $ (1/p,1/q)\in \mathcal{T}_{[1,2]}^{rad} $ only if $ p>p_d $. Furthermore, if $ M_{(0,\infty)} $ is $  L_{rad}^{p_d,1} \to L^{p_d,s}$ bounded, then $ s=\infty $. To see this, let $ f $ be the characteristic function of the unit ball in $ \R^d $, centered at the origin. Then $  M_{(0,\infty)}f(x)\geq A_{|x|}f(x)\gtrsim \frac{1}{|x|^{d-1}}$ when $ |x|\gg1 $. Thus, the radial restricted weak type $ (p_d,p_d) $ estimate for $ M_{(0,\infty)} $, which was proven by Leckband \cite{1987Leckband}, is the best possible Lorentz estimate.
\subsection{The case when $ \b<1$}
Before turning to the positive result of (i) in Theorem \ref{thmhigherdim}, we first introduce some useful notations. For $ n\geq 0 $, we define
\[W_n(E)=\{r\geq 0\,:\, \text{dist}(r, E)\leq 2^{-n+1}\}\]
and
\[D_n(E)=\{r\geq 0\,:\, 2^{-n}<\text{dist}(r, E)\leq 2^{-n+1}\}.\]
For convenience, we usually use the shorthand notations $ W_n $ and $ D_n $ where no ambiguity arises.
We denote the minimal number of binary  intervals of length $ 2^j $ (an interval of the form $ [m2^j, (m+1)2^j] $ for some $ m,k\in \mathbb{Z} $) needed to cover $ E $ by $ \tilde{N}(E, 2^j) $. We recall two useful lemmas in \cite[Lemma 2.2-2.3]{1997SWW}.
\begin{lem}\label{lem-binary}
	Let $ E\subset [1,2] $. Then, for $ n\geq 0 $,
	\[N(E,2^{-n})\leq \tilde{N}(E, 2^{-n})\leq 3N(E,2^{-n} )\]
	and
	\[2^{-n-2}N(E,2^{-n})\leq |W_n|\leq 2^{-n+3}N(E,2^{-n}).\]
\end{lem}
\begin{lem}\label{lem-equiv}
	Let $ E\subset [1,2]$ with $ |\bar{E}|=0$. Suppose that $ 0\leq \beta <1 $. Then the following three conditions are equivalent:
	\begin{align*}
		&\sup_{0<\de<1}\chi^E_{\mathrm{M},\beta}(\delta)<\infty,\\
		&\sup_{n\geq 0}|W_n|2^{n(1-\be)}<\infty,\\
		&\sup_{n\geq 0}|D_n|2^{n(1-\be)}<\infty.
	\end{align*}
\end{lem}
To prove the sufficient part of (i) in Theorem \ref{thmhigherdim}, it suffices to show the following proposition, according to Lemma \ref{lemR1}, Lemma \ref{lemR2} and (\ref{pointwise-higher}).
\begin{prop}\label{prop-beta}
	Let $ E\subset [1,2]$ and $ d\geq 2 $. Suppose that $ 0\leq \be<1 $ and 
	\begin{equation*}
	\sup_{0<\de<1}\chi^E_{\mathrm{M},\beta}(\delta)<\infty.
	\end{equation*}
	Then $ \mathfrak{M}_{E,p_0} $ is bounded from $ L^{p_0}(\R^+)$ to $L^{p_0}(\mu_d)$  and $ \mathfrak{M}_{E,p_1} $ is bounded from $ L^{p_1}(\R^+)$ to $L^{p_1d}(\mu_d)$, where $ p_0=1+\tfrac{\be}{d-1} $ and $ p_1=1+\tfrac{p_0}{d}$.
\end{prop}
\begin{proof}
	We first prove the estimate for $ \mathfrak{M}_{E,p_0} $. It is easy to see that
	\begin{align*}
		\|\mathfrak{M}_{E,p_0}g\|_{L^{p_0}(\mu_d)}\sim &\Big(\sum_{n\geq 0}\int_{D_n}|\mathfrak{M}_{E,p_0}g(r)|^{p_0}dr\Big)^{\tfrac{1}{p_0}}\\
		   +&\Big( \int_{r\in [2/3,4 ], \text{dist}(r,E)>2}|\mathfrak{M}_{E,p_0}g(r)|^{p_0}dr\Big)^{\tfrac{1}{p_0}}. 
	\end{align*}
The second term  is  easy to be dominated. Indeed, if $ r\in [2/3,4]$ and $\text{dist}(r,E)>2 $, then, by a simple application of H\"{o}lder inequality, we have
\begin{align*}
	\mathfrak{M}_{E,p_0}g(r)=&\sup_{\substack{t \in E \\ r/2<t<3r/2}}\int_{|r-t|}^{r+t}s^{\tfrac{d-1}{p_0'}-1}|g(s)|ds\\
	                       \lesssim&\int_{2}^{6}|g(s)|ds\lesssim (\int_{2}^{6}|g(s)|^{p_0}ds)^{\tfrac{1}{p_0}}\lesssim\|g\|_{L^{p_0}(\R^+)}.
	 \end{align*}
For the first term, by  dyadic decomposition, H\"{o}lder inequality and change of variables $ l=n-k $, we obtain that,  when $ r\in D_n $, 
\begin{align*}
	                       \mathfrak{M}_{E,p_0}g(r)&\leq \int_{2^{-n}}^{8}s^{\tfrac{d-1}{p_0'}-1}|g(s)|ds\\
	                       &=\sum_{k=0}^{n+2}\int_{2^{k-n}}^{2^{k-n+1}}s^{\tfrac{d-1}{p_0'}-1}|g(s)|ds\\
	                       &\leq \sum_{k=0}^{n+2}\big(\int_{2^{k-n}}^{2^{k-n+1}}s^{d-1-p_0'}ds\big)^{\tfrac{1}{p_0'}}\big(\int_{2^{k-n}}^{2^{k-n+1}}|g(s)|^{p_0}ds\big)^{\tfrac{1}{p_0}}\\
	                       &\lesssim \sum_{l=-2}^{n}2^{-l(\tfrac{d}{p_0'}-1)}\|g\|_{L^{p_0}(I_{-l})}.
\end{align*}
According to Lemma \ref{lem-equiv}, $|D_n|\lesssim 2^{-n(1-\be)} =2^{np_0\big(\tfrac{d}{p_0'}-1\big)}$. Consequently, 
	                       \begin{align*}
	                       \|\mathfrak{M}_{E,p_0}g\|_{L^{p_0}(\cup D_n;\mu_d)}&\lesssim\big(\sum_{n\geq 0}|D_n|\big( \sum_{l=-2}^{n}2^{-l(\tfrac{d}{p_0'}-1)}\|g\|_{L^{p_0}(I_{-l})}\big)^{p_0}\big)^{\tfrac{1}{p_0}}\\
	                       &\lesssim \big(\sum_{n\geq 0}\big( \sum_{l=-2}^{n}2^{(n-l)(\tfrac{d}{p_0'}-1)}\|g\|_{L^{p_0}(I_{-l})}\big)^{p_0}\big)^{\tfrac{1}{p_0}}.
	                       \end{align*}
We take the sequence $ a_k=\|g\|_{L^{p_0}(I_{-k})} $ for $ k\geq -2 $ and  $ a_k=0 $ for $ k\leq -2 $ . Another sequence $ b_k $ is defined by $ b_k=2^{k(\tfrac{d}{p_0'}-1)} $ for $ k\geq 0 $ and otherwise $ b_k=0 $. Note that $ \{b_k\}\in l^1 $ since $ p_0<\tfrac{d}{d-1} $. Using  Young's inequality for convolution of sequences, we have
\[\|\mathfrak{M}_{E,p_0}g\|_{L^{p_0}(\cup D_n;\mu_d)}\lesssim \big\|\sum_{l}a_{n-l}b_l\big\|_{l_n^{p_0}}\lesssim \|a_k\|_{l_k^{p_0}}\|b_k\|_{l_k^1}\lesssim \|g\|_{L^{p_0}(\R^+)}.\]
Next we  prove the estimate for $ \mathfrak{M}_{E,p_1} $. Similarly, it is enough to consider the $ L^{p_1d}(\mu_d) $ norm of $ \mathfrak{M}_{E,p_1}g $ on $ \cup D_n $. Proceeding as before, we have 
\[\mathfrak{M}_{E,p_1}g(r)\lesssim \sum_{l=-2}^{n}2^{-l(\tfrac{d}{p_1'}-1)}\|g\|_{L^{p_1}(I_{-l})},\]
when $ r\in D_n $.  Note that $|D_n|\lesssim 2^{-n(1-\be)} =2^{np_1d\big(\tfrac{d}{p_1'}-1\big)} $. Thus, we obtain that 
	                      \begin{align*}
	                       \|\mathfrak{M}_{E,p_1}g\|_{L^{p_1d}(\cup D_n;\mu_d)}&\lesssim\big(\sum_{n\geq 0}|D_n|\big( \sum_{l=-2}^{n}2^{-l(\tfrac{d}{p_1'}-1)}\|g\|_{L^{p_1}(I_{-l})}\big)^{p_1d}\big)^{\tfrac{1}{p_1d}}\\
	                       &\lesssim \big(\sum_{n\geq 0}\big( \sum_{l=-2}^{n}2^{(n-l)(\tfrac{d}{p_1'}-1)}\|g\|_{L^{p_1}(I_{-l})}\big)^{p_1d}\big)^{\tfrac{1}{p_1d}}.
	                       \end{align*}
 We take the similar sequence $ c_k=\|g\|_{L^{p_1}(I_{-k})} $ for $ k\geq -2 $ and  $ c_k=0 $ for $ k\leq -2 $. Set $ d_k=2^{k(\tfrac{d}{p_1'}-1)} $ for $ k\geq 0 $ and otherwise $ d_k=0 $. Then $ \{d_k\}\in l^1 $ since $ p_1=1+\tfrac{p_0}{d}<\tfrac{d}{d-1} $ as well. For simplicity, we define $e_n=\sum_{l}c_{n-l}d_l$. Using the natural embedding $ l^{p_1}\hookrightarrow l^{p_1d} $ and again Young's inequality for convolution of sequences, we have
\[\|\mathfrak{M}_{E,p_1}g\|_{L^{p_1d}(\cup D_n;\mu_d)}\lesssim \|e_n\|_{l_n^{p_1d}}\lesssim \|e_n\|_{l_n^{p_1}}\lesssim \|c_k\|_{l_k^{p_1}}\|d_k\|_{l_k^1}\lesssim \|g\|_{L^{p_1}(\R^+)},\]
which completes the proof.
\end{proof}
Using this proposition, we give a proof for the sufficient part of  (ii) in Theorem \ref{thmhigherdim}.
\begin{cor}
	Let $ d\geq 3 $ and $ E\subset [1,2]$  with $ \dim_{\mathrm{M}}\!E=\be <1$. Then   $ \De(\be)\setminus [Q_1(\be), Q_2(\be)]\subset  \mathcal{T}_E^{rad}$
\end{cor}
\begin{proof}
	The assumption $ \dim_{\mathrm{M}}\!E=\be <1$ means that $\sup_{0<\de<1}\chi^E_{\mathrm{M},\beta+\ep}(\delta)<\infty $ for any $ 0<\ep<1-\be $. Applying the above proposition, we have
	\[\bigcup_{\substack{0<\ep<1-\be }}\De(\be +\ep)\subset  \mathcal{T}_E^{rad},\]
	i.e. $ \De(\be)\setminus [Q_1(\be), Q_2(\be)]\subset  \mathcal{T}_E^{rad}$.
\end{proof}
\subsection{The endpoint case when $\b=1$ }
First, we recall the discrete Hardy's inequality. Although this inequality is well known, we give a proof for completeness.
\begin{lem}
	If  $ \{a_n\}_{n=1}^\infty \in l^p$ for some $ 1<p\leq \infty $, then
	\[\Big(\sum_{n\geq 1}\big|\frac{1}{n}\sum_{j=1}^{n}a_j\big|^p\Big)^{\tfrac{1}{p}}\lesssim \|a_n\|_{l^p_n}.\]
\end{lem}
\begin{proof}
	This lemma follows from interpolation between the weak type $ (1,1) $ estimate and the trivial  $ l^\infty\to l^\infty $ estimate. Note that $\tfrac{1}{n}|\sum_{j=1}^{n}a_j|\leq \tfrac{1}{n}\|a_j\|_{l_j^1}$, which gives $ l^1\to l^{1,\infty} $ estimate.
\end{proof}
Now  we pay attention to the positive result of (iii) in Theorem \ref{thmhigherdim}. As before, it suffices to prove the following proposition.
\begin{prop}
	Let $ E\subset [1,2] $. Suppose that $ w_n=|W_n| \leq n^{-\tfrac{q}{d}}$ where $  p_d=\tfrac{d}{d-1}\leq q<\infty$. Then
	\[\|\mathfrak{M}_{E,p_d}g\|_{L^{q}(\cup D_n;\mu_d)}\lesssim \|g\|_{L^{p_d}(\R^+)}.\]
\end{prop}
\begin{proof}
	Proceeding as before, we have, if $ r\in D_n $,   
	\[\mathfrak{M}_{E,p_d}g(r)\lesssim \sum_{k=-2}^{n}2^{-l(\tfrac{d}{p_d'}-1)}\|g\|_{L^{p_d}(I_{-k})}=\sum_{k=-2}^{n}\|g\|_{L^{p_d}(I_{-k})},\]
	since $ p_d'=d $. Set $a_n=\|g\|_{L^{p_d}(I_{-n})}$ and $ b_n=\sum_{k=1}^n a_k $. Using $ |D_n|=w_n-w_{n+1} $ and Abel transformation, we have
	\begin{align*}
		\sum_{n\geq 0}|D_n| b_n^q&=|D_0|b_0^q+\sum_{n\geq 1}(w_n-w_{n+1})b_n^q\\
		                         &=|D_0|b_0^q+\lim\limits_{m\to\infty}\sum_{n=1}^{m}(w_n-w_{n+1})b_n^q\\
		                         &=|D_0|b_0^q+w_1b_1^q+\lim\limits_{m\to\infty}[\sum_{n=2}^{m}w_n(b_n^q-b_{n-1}^q)-w_{m+1}b_m^q]\\
		                         &\leq |D_0|b_0^q+w_1b_1^q+\sum_{n\geq 1}w_{n+1}(b_{n+1}^q-b_{n}^q).
	\end{align*}
	Notice that $ b_{n+1}^q-b_{n}^q\lesssim (b_{n+1}-b_n)b_n^{q-1}=a_{n+1}b_{n+1}^{q-1} $. Then, 
	\[\|\mathfrak{M}_{E,p_d}g\|^q_{L^q(\cup D_n;\mu_d)}\leq \sum_{n\geq 0}|D_n| b_n^q\lesssim|D_0|b_0^q+w_1b_1^q+\sum_{n\geq 2}w_na_nb_n^{q-1}.\]
	It is easy to estimate the first two terms as before. For the last term, by H\"{o}lder inequality, we have
	\begin{align*}
		\sum_{n\geq 2}w_na_nb_n^{q-1}\leq & \|a_n\|_{l_n^{p_d}} \big(\sum_{n\geq 2} w_n^{d}b_n^{d(q-1)}\big)^{\tfrac{1}{d}}\\
		                              \lesssim &\|g\|_{L^{p_d}(\R^+)} \big(\sum_{n\geq 2} n^{-q}b_n^{d(q-1)}\big)^{\tfrac{1}{d}}.
	\end{align*}
	Note that 
	\[n^{-q}b_n^{d(q-1)}=\big(\frac{1}{n}b_n\big)^{p_d}\big(\frac{1}{n}b_n^d\big)^{q-p_d} \]
	 and again by H\"{o}lder inequality, 
	\[b_n=\sum_{k=-2}^n a_k\lesssim n^{\frac{1}{d}}\|a_k\|_{l_k^{p_d}}\lesssim  n^{\frac{1}{d}}\|g\|_{L^{p_d}(\R^+)}. \]
	Then, by an application of discrete Hardy's inequality, we have 
	\begin{align*}
		\sum_{n\geq 2} n^{-q}b_n^{d(q-1)}&\lesssim \|g\|_{L^{p_d}(\R^+)}^{d(q-p_d)}\sum_{n\geq 2}\big(\frac{1}{n}b_n\big)^{p_d}\\
		&\lesssim \|g\|_{L^{p_d}(\R^+)}^{d(q-p_d)}\|g\|_{L^{p_d}(\R^+)}^{p_d}=\|g\|_{L^{p_d}(\R^+)}^{d(q-1)},
	\end{align*}
	Combining all estimates together, we finish the proof.
\end{proof}

\section{$L^p$-improving bounds for circular maximal averages}\label{sect-circular}
In this section, we consider the $ L^p\to L^q $ estimates for the circular maximal operator on radial functions. The key difference is that the pointwise inequality (\ref{pointwise-higher}) is no longer applicable in two dimensions, since the kernel $ K_t(r,s) $ in (\ref{radialformula}) is much more singular in this case. Therefore, we need a more sophisticated treatment especially when proving Theorem \ref{thmcircular} and \ref{thm2Dendpoint}.
\subsection{The case when $ E=[1,2] $}
To obtain the radial restricted type estimates for $ M_{[1,2]} $, we exploit a variant of the  technical lemma proven by Nowak--Roncal--Szarek \cite[Lemma 4.3]{2023NowakRoncalSzarek}. Although the proof is quite similar, we provide it here for the reader's convenience.
\begin{lem}\label{lem-tech}
	Let $d=2 $ and $\tilde{\mathcal{A}}$ be a measurable set of $ \R^+ $. Define $ \mathcal{A}:=\{x\in\R^2\,:\, |x|\in \tilde{\mathcal{A}}\} $. Then, we have 
	\begin{equation}\label{ineq-circular}
		M_{[1,2]}\chi_{\mathcal{A}}(x)\lesssim [U\chi_{\tilde{\mathcal{A}}}(|x|)]^{\tfrac{1}{2}}+[R\chi_{\tilde{\mathcal{A}}}(|x|)]^{\tfrac{1}{2}},
	\end{equation}
	where the operators  $ U $ and $ R $ act on a function on $ \R^+ $ such that 
	\begin{align*}
		Uf(r)&=\chi_{(1/2,\infty)}(r)\sup_{\substack{t \in [1,2] \\ t<2r}}\frac{1}{r}\int_{|r-t|}^{r+t}zf(z)dz,\\
		Rf(r)&=\chi_{[0,1]}(r)\sup_{\substack{t \in [1,2] \\ t\geq 2r}}\frac{1}{r}\int_{t-r}^{t+r}f(z)dz.
	\end{align*}
\end{lem}
\begin{proof}
	Using (\ref{radialformula}) and (\ref{kernel}), we get
	\begin{equation}\label{equ-circularformula}
		A_t\chi_{\mathcal{A}}(x)=c_2\int_{A}^{B}[(z^2-A^2)(B^2-z^2)]^{-1/2}z\chi_{\tilde{\mathcal{A}}}(z)dz,
	\end{equation}
	where $ A=||x|-t| $ and $ B=|x|+t $. We claim that 
	\begin{equation}
		\int_{A}^{B}[(z^2-A^2)(B^2-z^2)]^{-1/2}z\chi_{\tilde{\mathcal{A}}}(z)dz\lesssim \Big(\frac{1}{B^2-A^2}\int_{A}^{B}z\chi_{\tilde{\mathcal{A}}}(z)dz\Big)^{1/2}.
	\end{equation}
	Then this lemma is  done by splitting the supremum into $ t<2r $ and $ t\geq 2r $.
	
	We divide the interval into  $ I_1=(A,(A+B)/2)\cap \tilde{\mathcal{A}}$ and $ I_2=((A+B)/2,B)\cap \tilde{\mathcal{A}}$. First, we consider the integral  over $ I_1 $. It is easy to see that our task is reduced to showing that 
	\begin{equation}\label{I1}
		\int_{I_1}[z/(z-A)]^{1/2}dz\lesssim \big(\int_{I_1}zdz\big)^{1/2}.
	\end{equation} 
	We consider two cases $ A\geq B/100 $ and $ A<B/100 $.\\
	\textbf{Case 1:} $ A\geq B/100 $.  In this case, $ z\sim A\sim B $ for $ z\in I_1 $. Since the function
	\[z\mapsto (z-A)^{-1/2}\]
	is strictly monotone decreasing when $ z>A $, we obtain that
	\begin{align*}
		\int_{I_1}[z/(z-A)]^{1/2}dz&\sim A^{1/2}\int_{I_1}(z-A)^{-1/2}dz\\
		                          &\leq A^{1/2}\int_{A}^{A+|I_1|}(z-A)^{-1/2}dz\\
		                          &\sim (A|I_1|)^{1/2}\sim\big(\int_{I_1}zdz\big)^{1/2}.
	\end{align*}
		\textbf{Case 2:} $ A< B/100 $. In this case, we split $ I_1 $ into $\tilde{\mathcal{A}_1}=(A,10A)\cap I_1$ and $\tilde{\mathcal{A}_2}=(10A, (A+B)/2)\cap I_1$, whose integrals for LHS of (\ref{I1}) are denoted by $ J_1 $ and $ J_2 $ respectively. Note that the required bound for $ J_1 $ follows directly from \textbf{Case 1} by taking $ B=19A $. When  $ z\in J_2 $ , we have $ z-A\sim z $. Set $ \mathcal{B}_k= \tilde{\mathcal{A}_2}\cap [2^k, 2^{k+1})$ for $ k\in \mathbb{Z} $. Then we get that
		\begin{align*}
			J_2^2\sim |\tilde{\mathcal{A}_2}|^2&=\big(\sum_{k}|\mathcal{B}_k|\big)^2\\
			                                   &\leq 2\sum_{k\in \mathbb{Z}} |\mathcal{B}_k|\big(\sum_{j\leq k}|\mathcal{B}_j|\big)\\
			                                   &\lesssim \sum_{k\in \mathbb{Z}} |\mathcal{B}_k|\big(\sum_{j\leq k}2^j\big)\\
			                                   &\sim \sum_{k\in \mathbb{Z}}2^k |\mathcal{B}_k|\sim \int_{\tilde{\mathcal{A}_2}}zdz.                              
		\end{align*}
		
		Now we focus on the integral over $ I_2 $. Proceeding as before, we only need to show that
		\[ \int_{I_2}(B-z)^{-1/2}dz\leq |I_2|^{-1/2}.\]
		 This can be proven similarly since  the function $ z\mapsto (B-z)^{-1/2} $ is strictly monotone increasing when $ z<B $.
\end{proof}
Then we have the following  easy estimates for $ U $ and $ R $.
\begin{lem}\label{lem-UR} 
	 The operator  $ U$ is bounded from $ L^1(\mu_2) $ to $ L^p(\mu_2) $  for any $ 1\leq p\leq \infty $, and the operator $ R $ is bounded from $ L^1(\mu_2) $ to $ L^q(\mu_2) $ for any $ 1\leq q<2 $. Moreover, $ R $ is of weak type $ (1,2) $, with respect to the measure $ \mu_2 $.
\end{lem}
\begin{proof}
	We first consider the operator $ U $. By interpolation, it is enough to show that $U$ is of strong type $ (1,1) $ and $ (1,\infty) $. Note that 
	\[Uf(r)\lesssim \frac{1}{r}\chi_{(1/2,\infty)}(r)\int_{\max\{0, r-2\}}^{r+2}|f(z)|zdz\lesssim\|f\|_{L^1(\mu_2)},\]
	which implies $ U $ is  bounded from $ L^1(\mu_2) $ to $ L^\infty(\mu_2)$. Furthermore, by Fubini's theorem, we have 
	\begin{align*}
		\int |Uf(r)|rdr&\leq \int_{1/2}^{\infty}\int_{\max\{0, r-2\}}^{r+2}|f(z)|zdz dr\\
		               &=\int_{0}^{\infty}\big(\int_{\max\{1/2, z-2\}}^{z+2}dr\big)|f(z)|zdz\\
		               &\lesssim \|f\|_{L^1(\mu_2)}.
	\end{align*}
	
	Now we focus on the operator $ R $. It is easy to see that
	\[Rf(r)\lesssim \frac{1}{r}\chi_{[0,1]}(r)\|f\|_{L^1(\mu_2)}.\]
	Consequently,  the desired bounds for $ R $ follow from direct calculation.
\end{proof}
With the help of  Lemma \ref{lem-tech} and Lemma \ref{lem-UR}, we can obtain some Lorentz estimates for the endpoints $ Q_1(1) $ and $ Q_2(1) $ in $ d=2 $, which by interpolation imply the positive results of localized circular maximal averages.
\begin{prop}
	Let $ d=2 $. Then $ M_{[1,2]} $ is bounded from  $ L_{rad}^{2,1}(\R^2) $  to  the Lorentz spaces $L^{2}(\R^2) $ and $L^{4,\infty}(\R^2) $.
\end{prop}
\begin{proof}
	We only show that $ M_{[1,2]} $ maps  $ L_{rad}^{2,1}(\R^2) $  boundedly into $L^{4,\infty}(\R^2)$; the proof for the radial restricted strong type $ (2,2) $ is similar. It suffices to show that 
	\[\lambda|\{x\in \R^2\, :\, 	M_{[1,2]}\chi_{\mathcal{A}}(x)>\lambda \}|^{1/4}\lesssim |\mathcal{A}|^{1/2},\]
	uniformly for $ \lambda>0 $, where $  \mathcal{A}$ is a radial set in $ \R^2 $ such that $ \mathcal{A}=\{x\in\R^2\,:\, |x|\in \tilde{\mathcal{A}}\} $ for some $ \tilde{\mathcal{A}}\subset \R^+ $. According to (\ref{ineq-circular}) and Lemma \ref{lem-UR}, we have, for any $\lambda>0$,
	\begin{align*}
		&\lambda|\{x\in\R^2\, :\, 	M_{[1,2]}\chi_{\mathcal{A}}(x)>\lambda \}|^{1/4}\\
\lesssim&\lambda \mu_2^{1/4}\big(\{r\,:\, |U\chi_{\tilde{\mathcal{A}}}(r)|^{1/2}>\lambda/2\}\big)+\lambda \mu_2^{1/4}\big(\{r\,:\, |R\chi_{\tilde{\mathcal{A}}}(r)|^{1/2}>\lambda/2\}\big)\\
=&\Big( \lambda^2\mu_2^{1/2}\big(\{r\,:\, |U\chi_{\tilde{\mathcal{A}}}(r)|>\lambda^2/4\}\big)\Big)^{1/2}+\Big( \lambda^2\mu_2^{1/2}\big(\{r\,:\, |R\chi_{\tilde{\mathcal{A}}}(r)|>\lambda^2/4\}\big)\Big)^{1/2}\\
\lesssim&\mu_2^{1/2}(\tilde{\mathcal{A}})=|\mathcal{A}|^{1/2},
	\end{align*}
	which completes the proof.
\end{proof}
Unlike the case in higher dimensions $ d\geq 3 $, it is surprising that the radial  restricted weak type estimate  at $ Q_2(1) $ in two dimensions is the best possible Lorentz estimate. Here, we present a counterexample that serves as a foundation for the proof of Theorem \ref{thmnecessary}.
\begin{prop}\label{prop-2DLorentz-necessary}
	Let $ d=2 $. Suppose that $ M_{[1,2]} $ is bounded from  $ L_{rad}^{2,r}(\R^2) $  to $ L^{4,s}(\R^2) $. Then $ r\leq 1 $ and $ s=\infty $.
\end{prop}
\begin{proof}
	As mentioned before(see the end of Section \ref{subsec localized in higher}), the necessary condition $ r\leq 1 $ follows from the standard Stein's counterexample. It remains to prove $ s=\infty $. We use a proof by contradiction.  Assume that $ s<\infty $. Let $ 0<\de\ll 1 $ and $ f_\de(x)=\chi_{[1-\de, 1]}(|x|) $. Then $ \|f_\de\|_{ L_{rad}^{2,1}(\R^2)}\sim\de^{1/2} $. By (\ref{radialformula}) and (\ref{kernel}), we have, if $ |x|\leq 1/4 $,
	\[M_{[1,2]}f_\de(x)\gtrsim r^{-1/2}\sup_{\substack{t \in [1,2] }}\int_{t-r}^{t}\chi_{[1-\de, 1]}(z)(z-(t-r))^{-1/2}dz ,\]
	where $ r=|x|$. For any $ \de\leq r \leq 1/4$, choose $ t=r+1-\de $. Then we have
	\[M_{[1,2]}f_\de(x)\geq c r^{-1/2}\int_{1-\de}^{1}(z-(1-\de))^{-1/2}dz=2c\de^{1/2}r^{-1/2} .\]
	Here $ c $ is  a suitable constant. Consequently, when $c \de^{1/2}\leq \lambda \leq c$, 
	\[\lambda|\{x\in\R^2\,:\, M_{[1,2]}f_\de(x)>\lambda\}|^{1/4}\gtrsim \de^{1/2}.\]
	Then we obtain 
	\[\|M_{[1,2]}f_\de\|^s_{L^{4,s}(\R^2)}\gtrsim \int_{c\de^{1/2}}^{c}\de^{s/2}\dfrac{d\la}{\la}\sim \de^{s/2}\log(1/\de).\]
	This implies $ \de^{s/2}\log(1/\de)\lesssim \de^{s/2} $, which is a contradiction after letting $ \de\to 0 $.
\end{proof}
\subsection{The case when $ \b<1 $ }
As before, we need a pointwise  inequality for $ M_E $ acting on radial function in $ \R^2 $ with $ f(x)=f_0(r) $ where $ r=|x|$. One possible approach is to generalize Lemma \ref{lem-tech} by replacing the interval $ [1,2] $ in the definitions of $ U $ and $ R $ with the fractal set $ E $. However, we do not know how to obtain the corresponding $ L^p\to L^q $ estimates for the maximal function over a general fractal set when $ p<1 $. Instead, we proceed with the decomposition from \cite{1997SWW}.
\begin{lem}[\cite{1997SWW}, Lemma 5.1]
	Let $ d=2 $, $ E\subset[1,2] $ and $ f $ is a radial function defined in $ \R^2 $. Then
	\begin{equation}\label{pointwise-2D}
	M_Ef(x)\lesssim [\mathfrak{M}_{E,p}g(r)+\widetilde{\mathfrak{M}}_{E,p}g(r)+\sum_{i=1}^{4}R_{i,E}f_0(r)],
	\end{equation}
	where 
	\[ g(s)=f_0(s)s^{\tfrac{1}{p}},\]
	\begin{align*}
	\mathfrak{M}_{E,p}g(r)&=\chi_{(2/3,4)}(r)\sup_{\substack{t \in E \\ r/2<t<3r/2}}\int_{|r-t|}^{r+t}s^{1/2-1/p}(s-|r-t|)^{-1/2}|g(s)|ds,\\
	\widetilde{\mathfrak{M}}_{E,p}g(r)&=\chi_{(2/3,4)}(r)\sup_{\substack{t \in E \\ r/2<t<3r/2}}\int_{|r-t|}^{r+t}s^{1/2-1/p}(r+t-s)^{-1/2}|g(s)|ds,
	\end{align*}
	and
	\begin{align*}
	R_{1,E}f_0(r)&=\chi_{[2,\infty)}(r)\sup_{\substack{t \in E \\ t\leq r/2}}\int_{r-t}^{r}(s-r+t)^{-1/2}|f_0(s)|ds,\\
	R_{2,E}f_0(r)&=\chi_{[2,\infty)}(r)\sup_{\substack{t \in E \\ t\leq r/2}}\int_{r}^{r+t}(r+t-s)^{-1/2}|f_0(s)|ds,\\
	R_{3,E}f_0(r)&=\chi_{[0,4/3)}(r)\sup_{\substack{t \in E \\ t\geq 3r/2}}r^{-1/2}\int_{t-r}^{t}(s-t+r)^{-1/2}|f_0(s)|ds,\\
	R_{4,E}f_0(r)&=\chi_{[0,4/3)}(r)\sup_{\substack{t \in E \\ t\geq 3r/2}}r^{-1/2}\int_{t}^{t+r}(r+t-s)^{-1/2}|f_0(s)|ds.
	\end{align*}
\end{lem}
When presenting the above lemma, we closely follow the notation used in \cite{1997SWW}.  $ \mathfrak{M} $ and $ \widetilde{\mathfrak{M}}$ denote the main term while $ R_i$ denotes the remainder term, which is indeed the case when proving the $ L^p\to L^p $ estimates for circular maximal operator over $ E\subset(0,\infty) $. However, when it comes to  the $ L^p\to L^q $ maximal estimates over restricted dilation sets $ E\subset[1,2] $, the situation undergoes a subtle change. In some cases (see Proposition \ref{prop2D-R3} below), $ R_{3,E} $ and $ R_{4,E} $ may become the main term as well. Nevertheless, we continue to use the original notation to make it easier for readers to find the connection between our paper and \cite{1997SWW}.

We first prove $ L^p-$improving estimates for $ R_{1,E} $ and $R_{2,E}$.
\begin{prop}\label{prop2D-R1}
	Let $ d=2 $ and $ E\subset [1,2] $ with $ \dim_{\mathrm{M}}\!E=\be $. Then, for $ 1\leq p\leq q\leq\infty $ satisfying $ \frac{1}{p}-\frac{1-\be}{q}-\frac{1}{2}<0 $, we have 
	\[\|R_{i,E}f_0\|_{L^q(\mu_2)}\leq\|f_0\|_{L^p(\mu_2)} ,\]
	for $ i=1, 2 $.
\end{prop}
\begin{proof}
	We prove the estimates for $ R_{1,E} $ only, since the  proof  for $ R_{2,E} $  is similar. For $ k\geq 1 $ and $ m\geq 0 $, we define
	\[R_{1,E}^{k,m}f_0(r)=2^{-k+\tfrac{m}{2}}\chi_{I_k}(r)\sup_{\substack{t \in E \\ t\leq r/2}}\int_{r-t+2^{-m-1}t}^{r-t+2^{-m}t}|f_0(s)|sds\]
	Then it is easy to see that 
	\begin{equation}\label{ineq-R1sum}
		\|R_{1,E}f_0\|_{L^q(\mu_2)}\leq \big(\sum_{k\geq 1}\|\sum_{m\geq 0}R_{1,E}^{k,m}f_0\|^q_{L^q(\mu_2)}\big)^{1/q}.
	\end{equation}
	We claim that, for $ 1\leq p\leq q\leq\infty $,
	\begin{equation}\label{R1-Lpq}
		\|R_{1,E}^{k,m}f_0\|_{L^q(\mu_2)}\lesssim 2^{-k(\tfrac{1}{p}-\tfrac{1}{q})}2^{m(\tfrac{1}{p}-\tfrac{1}{q}-\tfrac{1}{2})}N(E,2^{-m})^{\tfrac{1}{q}}\|f_0\|_{L^p(\tilde{I}_k;\mu_2)}.
	\end{equation}
This inequality follows from interpolation between $ (p,q)=(\infty, \infty), (1,\infty)$ and $ (1,1) $. The $ L^1\to L^\infty $ and $ L^\infty\to L^\infty $ estimates are straightforward, so we omit the details. For the $ L^1\to L^1 $ estimate, we cover $ E $ with intervals  $ I_\nu^m $ of length $ 2^{-m} $. Note that the number of these intervals are comparable to $ N(E,2^{-m}) $. We denote the double interval of $ I_\nu^m $ by $ \tilde{I}_\nu^m$. Then, by Fubini's theorem, we obtain that
\begin{align*}
	\int R_{1,E}^{k,m}f_0(r)rdr&\lesssim 2^{m/2}\sum_\nu\int_{2^k}^{2^{k+1}}\int_{r-\tilde{I}_\nu^m}|f_0(s)|sdsdr\\
	                           &\leq 2^{m/2}\sum_\nu\int_{\tilde{I}_k}\big(\int_{s+\tilde{I}_\nu^m}dr\big)|f_0(s)|sds\\
	                           &\lesssim 2^{-m/2}N(E,2^{-m})\|f_0\|_{L^1(\tilde{I}_k;\mu_2)}.
\end{align*}
Notice that $ N(E,2^{-m})\lesssim_{\ep} 2^{m(\be+\ep)}$ for any $ \ep>0$. Then, if  $ \frac{1}{p}-\frac{1-\be}{q}-\frac{1}{2}<0 $, we have
\[\|\sum_{m\geq 0}R_{1,E}^{k,m}f_0\|_{L^q(\mu_2)}\lesssim 2^{-k(\tfrac{1}{p}-\tfrac{1}{q})}\|f_0\|_{L^p(\tilde{I}_k;\mu_2)},\]
by direct summation and (\ref{R1-Lpq}). Using the natural embedding $ l^{p}\hookrightarrow l^{q} $ for $ p\leq q $ and (\ref{ineq-R1sum}), we have
\begin{align*}
	\|R_{1,E}f_0\|_{L^q(\mu_2)}&\leq \big(\sum_{k\geq 1}\|\sum_{m\geq 0}R_{1,E}^{k,m}f_0\|^p_{L^q(\mu_2)}\big)^{1/p}\\
	                          &\lesssim \big(\sum_{k\geq 1}\|f_0\|^p_{L^p(\tilde{I}_k;\mu_2)}\big)^{1/p}\\
	                          &\lesssim \|f_0\|_{L^p(\mu_2)},
\end{align*}
as long as $ 1\leq p\leq q\leq\infty $ and $ \frac{1}{p}-\frac{1-\be}{q}-\frac{1}{2}<0 $.
\end{proof}
With the help of the above proposition, we can easily get the following corollary when the Minkowski dimension of $ E $ is strictly less than 1.
\begin{cor}\label{cor2D-R1}
	Suppose that $\dim_{\mathrm{M}}\!E=\be<1$. Then, for $ i=1, 2 $, $ R_{i,E} $ is bounded from $L^p(\mu_2) $ to $ L^q(\mu_2) $ if $ (1/p, 1/q)\in \De(\be) $.
\end{cor}
Then we turn to the operators $\mathfrak{M}_{E,p}$ and $\widetilde{\mathfrak{M}_{E,p}} $, which require a more refined decomposition.
\begin{prop}\label{prop2D-M}
	Let $ d=2 $ and $  E\subset [1,2] $.  Suppose that $ 0\leq \be<1 $ and 
	\begin{equation*}
	\sup_{0<\de<1}\chi^E_{\mathrm{M},\beta}(\delta)<\infty.
	\end{equation*}
	Then
	\[\|\mathfrak{M}_{E,p_0}\|_{L^{p_0}(\R^+)\to L^{p_0}(\mu_2)}+\|\widetilde{\mathfrak{M}}_{E,p_0}\|_{L^{p_0}(\R^+)\to L^{p_0}(\mu_2)}<\infty,\]
	and
	\[\|\mathfrak{M}_{E,p_1}\|_{L^{p_1}(\R^+)\to L^{2p_1}(\mu_2)}+\|\widetilde{\mathfrak{M}}_{E,p_1}\|_{L^{p_1}(\R^+)\to L^{2p_1}(\mu_2)}<\infty,\]
	where $ p_0=1+\be $ and $ p_1=1+\tfrac{p_0}{2}$.
\end{prop}
\begin{proof}
	We only consider $ \mathfrak{M}_{E,p_i}$ with $ i=0,1 $ since the proof for $ \widetilde{\mathfrak{M}}_{E,p_i} $ is similar. Define
	\[\mathfrak{M}_{E,p}^0g(r)=\chi_{(2/3,4)}(r)\sup_{\substack{t \in E \\ r/2<t<3r/2}}\int_{|r-t|}^{2|r-t|}s^{1/2-1/p}(s-|r-t|)^{-1/2}|g(s)|ds,\]
	and
	\[\mathfrak{M}_{E,p}^1g(r)=\chi_{(2/3,4)}(r)\sup_{\substack{t \in E \\ r/2<t<3r/2}}\int_{2|r-t|}^{r+t}s^{-1/p}|g(s)|ds.\]
	Obviously, we have
	\[\mathfrak{M}_{E,p}g(r)\lesssim \mathfrak{M}_{E,p}^0g(r)+\mathfrak{M}_{E,p}^1g(r).\]
	By Proposition \ref{prop-beta}, we obtain the desired estimates 
	\[\|\mathfrak{M}_{E,p_0}^1\|_{L^{p_0}(\R^+)\to L^{p_0}(\mu_2)}+\|\mathfrak{M}_{E,p_1}^1\|_{L^{p_1}(\R^+)\to L^{2p_1}(\mu_2)}\leq \infty.\]
	So it remains to handle $ \mathfrak{M}_{E,p}^0 $. For simplicity, we only show $L^{p_1}(\R^+)\to L^{2p_1}(\mu_2) $ estimate for the operator $ \mathfrak{M}_{E,p_1}^0 $ since
	the corresponding bound for $\mathfrak{M}_{E,p_0}^0 $ is very similar.  Actually, a stronger conclusion for $L^{p_0}(\R^+)\to L^{p_0}(\mu_2) $ estimate is included in \cite[Proposition 5.4]{1997SWW}.
	
	Define  $\tilde{D}_n=\{r\geq 0\,:\, 2^{-n}<\text{dist}(r, E)< 2^{-n+1}\}$. Then we can easily get
	\begin{equation}\label{newDn-size}
		|D_n|=|\tilde{D}_n| \lesssim 2^{n(1-\be)}=2^{2p_1n(1-\tfrac{2}{p_1})},
	\end{equation} 
	since $ |\bar{E}|=0 $. Note that the open set $ \tilde{D}_n $ can be written as a union of intervals $ I_{n,\nu} $. We divide $ E $ into 
	\[E_{n,l,\nu}:=\{t\in E\,:\, 2^{-n+l}\leq \text{dist}(t, I_{n,\nu})\leq 2^{-n+l+1}\},\]
	for $ 0\leq l\leq n$. By H\"{o}lder inequality, we obtain 
	\[\mathfrak{M}_{E,p_1}^0g(r)\lesssim \sum_{l\geq 0}\sum_{m\geq 0}\mathfrak{M}_{E,p_1}^{l,m}g(r),\]
	where
	\[\mathfrak{M}_{E,p_1}^{l,m}g(r)=\sum_{n\geq l}\sum_{\nu}\chi_{I_{n,\nu}}(r)2^{(2n-2l+m)(\tfrac{1}{p_1}-\tfrac{1}{2})}\sup_{\substack{t \in E_{n,l,\nu} }}\Big(\int_{|r-t|+2^{-n+l-m}}^{|r-t|+2^{-n+l-m+1}}|g(s)|^{p_1}ds\Big)^{1/p_1}.\]
	It suffices to show that 
	\begin{equation}\label{est-m-small}
		\big(\sum_{n\geq 0}\|\sum_{l\geq 0}\sum_{m<l}\mathfrak{M}_{E,p_1}^{l,m}g\|_{L^{2p_1}(\tilde{D}_n;\mu_2)}\big)^{\tfrac{1}{2p_1}}\lesssim \|g\|_{L^{p_1}(\R^+)},
	\end{equation}
	and
	\begin{equation}\label{est-m-large}
		\big(\sum_{n\geq 0}\|\sum_{l\geq 0}\sum_{m\geq l}\mathfrak{M}_{E,p_1}^{l,m}g\|_{L^{2p_1}(\tilde{D}_n;\mu_2)}\big)^{\tfrac{1}{2p_1}}\lesssim \|g\|_{L^{p_1}(\R^+)}.
	\end{equation}
	To get (\ref{est-m-small}), we use the crude estimate
	\begin{equation}\label{crude}
		\|\mathfrak{M}_{E,p_1}^{l,m}g\|_{L^\infty(\tilde{D}_n;\mu_2)}\lesssim2^{(2n-2l+m)(\tfrac{1}{p_1}-\tfrac{1}{2})}\|g\|_{L^{p_1}(\tilde{I}_{-n+l})}.
	\end{equation}
	Consequently, by Minkowski inequality, (\ref{crude}) and (\ref{newDn-size}), we have
	\begin{align*}
		\|\sum_{l\geq 0}\sum_{m< l}\mathfrak{M}_{E,p_1}^{l,m}g\|_{L^{2p_1}(\tilde{D}_n;\mu_2)}&\lesssim \sum_{l\geq 0}\sum_{m< l}\| \mathfrak{M}_{E,p_1}^{l,m}g\|_{L^{2p_1}(\tilde{D}_n;\mu_2)}\\
		&\lesssim \sum_{l\geq 0}|\tilde{D}_n|^{\tfrac{1}{2p_1}}2^{(2n-2l)(\tfrac{1}{p_1}-\tfrac{1}{2})}\|g\|_{L^{p_1}(\tilde{I}_{-n+l})}\sum_{m< l}2^{m(\tfrac{1}{p_1}-\tfrac{1}{2})}\\
		&\lesssim \sum_{l\geq 0}2^{-l(\tfrac{1}{p_1}-\tfrac{1}{2})}\|g\|_{L^{p_1}(\tilde{I}_{-n+l})}.
	\end{align*}
	Notice that $ p_1<2 $ since $ \be<1 $. As in the proof of Proposition \ref{prop-beta}, based on the embedding $ l^{p_1}\hookrightarrow l^{2p_1} $ and Young's inequality for convolution of sequences, we obtain the required estimate (\ref{est-m-small}). For (\ref{est-m-large}), we need the finer estimate
	\begin{equation}\label{finer}
		\|\mathfrak{M}_{E,p_1}^{l,m}g\|_{L^{p_1}(\tilde{D}_n;\mu_2)}\lesssim 2^{\tfrac{l-n}{p_1'}}2^{-\tfrac{m}{2}}[N(E,2^{-n+l-m-8})]^{1/p_1}\|g\|_{L^{p_1}(\tilde{I}_{-n+l})},
	\end{equation}
	which had already appeared in \cite[(5.11)]{1997SWW}. Combining (\ref{crude}) and (\ref{finer}) with H\"{o}lder inequality, we get
	\[\|\mathfrak{M}_{E,p_1}^{l,m}g\|_{L^{2p_1}(\tilde{D}_n;\mu_2)}\lesssim  2^{-m(\tfrac{1}{p_1}-\tfrac{1}{2})}\|g\|_{L^{p_1}(\tilde{I}_{-n+l})}.\]
	Thus, 
	\[\|\sum_{l\geq 0}\sum_{m\geq l}\mathfrak{M}_{E,p_1}^{l,m}g\|_{L^{2p_1}(\tilde{D}_n;\mu_2)}\lesssim \sum_{l\geq 0}2^{-l(\tfrac{1}{p_1}-\tfrac{1}{2})}\|g\|_{L^{p_1}(\tilde{I}_{-n+l})}.\]
	We proceed analogously to get (\ref{est-m-large}).
\end{proof}
The following corollary is a direct consequence if  we replace $ \be $ with $ \be+\ep $ in the assumption about the Minkowski characteristic of $ E $ for any $0<\ep<1-\be$. 
\begin{cor}\label{cor2D-M}
		Let $ d=2 $ and $  E\subset [1,2] $ with $ \dim_{\mathrm{M}}\!E=\be<1$. Then, for $ (1/p,1/q) \in \De(\be)\setminus [Q_1(\be), Q_2(\be)]$, we have
		\[\|\mathfrak{M}_{E,p}g\|_{L^q(\mu_2)}+\|\widetilde{\mathfrak{M}}_{E,p}g\|_{L^q(\mu_2)}\lesssim \|g\|_{L^p(\R^+)}.\]
\end{cor}

To establish the endpoint result, we also need to consider the operators $R_{3,E} $ and $R_{4,E} $. Surprisingly, the region of $ (1/p,1/q) $ such that these two operators  are $ L^p\to L^q $ bounded depends on the Assouad dimension. 

Before stating the proposition, we introduce some additional notations. Let $\tilde{Q}_3(\be)=(\frac{2-\be}{2},\frac{1}{2})$ and $\tilde{\mathcal{Q}}(\be, \ga) $ denote the closed quadrilateral formed by vertices $ O$, $Q_1(\be)$, $\tilde{Q}_3(\be)$, $Q_2(2\ga-1)$.  Note that the quadrilateral $ \tilde{\mathcal{Q}}(\be, \ga)$   degenerates into a triangle  when $\be=\ga=0$.

\begin{prop}\label{prop2D-R3}
		Let $ d=2 $ and $ E\subset [1,2] $, and $ 0\leq \be\leq \ga\leq 1 $. Suppose that $ \sup_{0<\de<1}\chi^E_{\mathrm{M},\beta}(\delta)<\infty $ and  $\sup_{0<\de<1}\chi^E_{\mathrm{A},\ga}(\delta)<\infty $.
	
	(i) If $0<\be<1 $, then   $ R_{i,E} $ is bounded from $L^p(\mu_2) $ to $L^q(\mu_2)$ for $(1/p,1/q)\in \tilde{\mathcal{Q}}(\be, \ga)\setminus\{Q_2(2\ga-1),\tilde{Q}_3(\be)\}$ and $ i=3,4$. Moreover, $ R_{i, E} $ $ (i=3,4)$ maps $ L^{p,1}(\mu_2) $ boundedly into  $ L^{q,\infty}(\mu_2) $ when  $(1/p,1/q)$ is  the point $ Q_2(2\ga-1) $ or $ \tilde{Q}_3(\be) $.
	
	(ii) If $\be=0 $, then $ R_{i, E} $ is bounded from $L^p(\mu_2) $ to $L^q(\mu_2)$ for $(1/p,1/q)\in \tilde{\mathcal{Q}}(0,0)\setminus (Q_1(0),\tilde{Q}_3(0)] $ and $ i=3,4 $. 
\end{prop}
\begin{proof}[Proof of  Theorem \ref{thm2Dendpoint}]
	Theorem \ref{thm2Dendpoint} is immediate from a combination of Proposition \ref{prop2D-R1}, \ref{prop2D-M} and \ref{prop2D-R3}. Note that $ \Delta(\beta)\subset \tilde{\mathcal{Q}}(\be, \ga)$ when $ 2\ga<\be+1 $ and $ \mathcal{Q}(\be, \ga)= \Delta(\beta)\cap \tilde{\mathcal{Q}}(\be, \ga)$ when $ 2\ga\geq \be+1 $.
\end{proof}
To get the restricted weak type estimate for the endpoint situation, we will apply an interpolation argument due to  Bourgain \cite{1985Bourgain}. See \cite[Section 6.2]{1999CSWW} for an abstract extension of Bourgain's interpolation trick.
\begin{lem}
	Suppose that $ a_0,a_1>0 $ and  $ (T_j)_{j\geq 0} $ are sublinear operators such that
	\[\|T_j\|_{L^{p_0,1}(\mu)\to L^{q_0,\infty}(\mu) }\leq M_12^{ja_0}, \quad \|T_j\|_{L^{p_1,1}(\mu)\to L^{q_1,\infty}(\mu) }\leq M_22^{-ja_1},\]
	uniformly for $ j\geq 0 $. Then $ \sum_{j\geq 0}T_j $ maps $ L^{p,1}(\mu) $ to $ L^{q,\infty}(\mu) $ with the operator norm $ O(M_1^{1-\theta}M_2^\theta) $, where
	\[(\frac{1}{p},\frac{1}{q})=(1-\theta)(\frac{1}{p_0},\frac{1}{q_0})+\theta(\frac{1}{p_1},\frac{1}{q_1}),\quad \theta=\frac{a_0}{a_0+a_1}.\]
\end{lem}

\begin{proof}[Proof of Proposition \ref{prop2D-R3}]
	We only show the estimates for $ R_{3,E} $ since the corresponding estimates for $ R_{4,E} $ can be proven in the same way. Define
	\[R_{3,E}^{k,m}f_0(r)=2^{k+\frac{m}{2}}\chi_{I_{-k}}(r)\sup_{\substack{t \in E \\ t\geq 3r/2}}\int_{t-r+2^{-m-1}r}^{t-r+2^{-m}r}|f_0(s)|sds,\]
	for $ k\geq 0 $ and $ m\geq 0 $. Then $R_{3,E}f_0(r)\lesssim \sum_{k\geq 0}\sum_{m\geq 0}R_{3,E}^{k,m}f_0(r)$. We claim that, for $ 1\leq p\leq q\leq \infty $,
	\begin{equation}\label{keyestimates for R3}
		\|R_{3,E}^{k,m}f_0\|_{L^q(\mu_2)}\lesssim [\sup_{\substack{ |I|=2^{-k}}}N(E\cap I, 2^{-m-k})]^{\tfrac{1}{q}}2^{m(\tfrac{1}{p}-\tfrac{1}{q}-\tfrac{1}{2})}2^{k(\tfrac{1}{p}-\tfrac{2}{q})}\|f_0\|_{L^p(\mu_2)}.
	\end{equation}
By interpolation, it suffices to show the inequality for $(p,q)=(\infty, \infty), (1,\infty)$ and $ (1,1) $. The $ L^1\to L^\infty $ and $ L^\infty\to L^\infty $ estimates are straightforward, so we leave it to the reader. For the $ L^1\to L^1 $ estimate, we cover $ E $ with intervals  $I_\nu^{m+k}$ of length $ 2^{-m-k} $.	Note that 
\begin{align*}
	R_{3,E}^{k,m}f_0(r)&\lesssim2^{k+\frac{m}{2}}\chi_{I_{-k}}(r)\sum_{\nu}\sup_{\substack{t\in E\cap I_\nu^{k+m}\\ t\geq 3r/2}}\int_{t-r+2^{-m-1}r}^{t-r+2^{-m}r}|f_0(s)|sds\\
	                  &\lesssim2^{k+\frac{m}{2}}\chi_{I_{-k}}(r)\sum_{\nu}\int_{-r+\tilde{I}_\nu^{k+m}}|f_0(s)|sds.
\end{align*}
Here we denote the double interval of $I_\nu^{m+k}$ by $ \tilde{I}_\nu^{k+m}$. Then, by applying Fubini's theorem twice, we have
\begin{align*}
	\int_{2^{-k}}^{2^{-k+1}}|R_{3,E}^{k,m}f_0(r)|rdr&\lesssim 2^{\frac{m}{2}}\sum_{\nu}\int_{I_{-k}}\int_{-r+\tilde{I}_\nu^{k+m}}|f_0(s)|sdsdr\\
	                                                &=2^{\frac{m}{2}}\sum_{\nu}\int_{\tilde{I}_\nu^{k+m}-I_{-k}}\big(\int_{-s+\tilde{I}_\nu^{k+m}}dr\big)|f_0(s)|sds\\
	                                                &\lesssim 2^{-\frac{m}{2}-k}\sum_{\nu}\int_{\tilde{I}_\nu^{k+m}-I_{-k}}|f_0(s)|sds\\
	                                                &\lesssim 2^{-\frac{m}{2}-k}\int_{0}^\infty |f_0(s)|s\big(\sum_{\nu\in A^{k,m}(s)}1\big)ds,
\end{align*}
where $ A^{k,m}(s):=\{\nu\, :\, s\in \tilde{I}_\nu^{k+m}-I_{-k}\}$. For any $ \nu_1, \nu_2\in A^{k,m}(s)$, it is easy to see that $ \text{dist}(\tilde{I}_{\nu_1}^{k+m},\tilde{I}_{\nu_2}^{k+m})\lesssim 2^{-k} $. This leads us to the uniform  upper bound 
\[\#A^{k,m}(s)\lesssim \sup_{\substack{ |I|=2^{-k}}}N(E\cap I, 2^{-m-k}).\]
Therefore, the required $ L^1(\mu_2)\to L^1(\mu_2) $ estimate has been completed.

We first focus on the case $ 0<\be<1$. The $  L^{\infty}(\mu_2)\to L^{\infty}(\mu_2) $ estimate is a direct consequence of calculation. Thus, it suffices to show the strong type estimate at $ Q_1(\beta) $ and the restricted weak type estimates at $Q_2(2\ga-1)$ and $\tilde{Q}_3(\be) $. Then the remaining strong type estimates follow easily from interpolation. By the assumption about  the  $\beta$-Minkowski  and $\ga$-Assouad characteristic of $ E $, we have
\begin{equation}\label{gamma-covering}
	\sup_{\substack{ |I|=2^{-k}}}N(E\cap I, 2^{-m-k})\lesssim 2^{m\ga},
\end{equation}
and 
\begin{equation}\label{beta-covering}
	\sup_{\substack{ |I|=2^{-k}}}N(E\cap I, 2^{-m-k})\leq N(E,2^{-m-k} )\lesssim 2^{(m+k)\be}.
\end{equation}
Substituting (\ref{beta-covering}) into (\ref{keyestimates for R3}), we get
\[\|R_{3,E}^{k,m}f_0\|_{L^q(\mu_2)}\lesssim2^{m(\tfrac{1}{p}-\tfrac{1-\be}{q}-\tfrac{1}{2})}2^{k(\tfrac{1}{p}-\tfrac{2-\be}{q})}\|f_0\|_{L^p(\mu_2)}.\]
By direct summation, $ L^{1+\be}(\mu_2)\to L^{1+\be}(\mu_2)$ estimate for $ R_{3,E} $ is obtained. Note that $\tilde{Q}_3(\be)$ is the intersection point of these two lines $ \tfrac{1}{p}-\tfrac{1-\be}{q}-\tfrac{1}{2}=0$ and $ \tfrac{1}{p}-\tfrac{2-\be}{q}=0 $. Thus, the restricted weak type estimate at $\tilde{Q}_3(\be)$ follows from applying Bourgain's interpolation trick twice. For the endpoint estimate at $ Q_2(2\ga-1)$, we substitute (\ref{gamma-covering}) into (\ref{keyestimates for R3}) to get
\[\|R_{3,E}^{k,m}f_0\|_{L^q(\mu_2)}\lesssim2^{m(\tfrac{1}{p}-\tfrac{1-\ga}{q}-\tfrac{1}{2})}2^{k(\tfrac{1}{p}-\tfrac{2}{q})}\|f_0\|_{L^p(\mu_2)}.\] 
Again, by using Bourgain's interpolation lemma twice, $ R_{3,E} $ is of the restricted weak type at $ Q_2(2\ga-1)$. 

Now we consider the simple case $ \be=0 $. Notice that $ \sup_{0<\de<1}\chi^E_{\mathrm{M},0}(\delta)<\infty$  implies that there are only finite points in $ E $ and $ \ga=0 $ as well. The only difference is  that the point $\tilde{Q}_3(0)=(1,1/2)$  does not lie in the interior of the triangle formed by $(0,0), (1,1), (1,0)$, so we can not obtain the restricted weak type at $ \tilde{Q}_3(0) $.
Aside from this, the other proofs are analogous to the previous case.
\end{proof}
The above proof gives non-endpoint results for $ R_{3,E} $ and $ R_{4,E} $.
\begin{cor}\label{cor2D-R3}
	Let $ d=2 $ and $ E\subset [1,2] $ with $ \dim_{\mathrm{M}}\!E=\be<1$, $\dim_{\mathrm{qA}}\!E=\ga$, $\dim_{\mathrm{A}}\!E=\ga_*$. Then $ R_{i,E} $ is bounded from $ L^p(\mu_2) $  to $ L^q(\mu_2) $ for $ i=3,4 $ and
	\[(1/p,1/q)\in \text{int}(\tilde{\mathcal{Q}}(\be, \ga))\cup[O, Q_1(\be)\big)\cup \big(O, Q_2( 2\ga_*-1)\big).\]
\end{cor}
\begin{proof}
	For  small enough $ \ep>0 $, we have
	\[\sup_{0<\de<1}\chi^E_{\mathrm{M},\be+\ep}(\delta)<\infty,\quad  \sup_{0<\de<1}\chi^E_{\mathrm{A},\ga_*+\ep}(\delta)<\infty.\]
	Applying Proposition \ref{prop2D-R3}, we obtain the strong type estimates on $[O,Q_1(\be)) $ and $(O.Q_2(2\ga_*-1)) $. According to the definition of $\dim_{\mathrm{qA}}\!E $,  we have
	\begin{equation}\label{loss-covering}
		N(E\cap I,\de)\lesssim_\ep(\de/|I|)^{-\ga-\ep}\de^{-\ep},
	\end{equation}
	for any $ 0<\ep\ll1 $ and $ |I|\geq \de $. Indeed, one can easily get this by considering the case $|I|\geq \de^{1-\ep} $ and $\de\leq |I|\leq \de^{1-\ep}$ respectively.
	Using (\ref{loss-covering}), we proceed analogously as in the proof of Proposition \ref{prop2D-R3} to get the strong type estimates in the interior of $\tilde{\mathcal{Q}}(\be, \ga)$.
\end{proof}
\begin{proof}[Proof of Theorem \ref{thmcircular}]
	When $ \be<1 $, Theorem \ref{thmcircular} is an immediate consequence of Corollary \ref{cor2D-R1}, \ref{cor2D-M} and \ref{cor2D-R3}. The $ \be=1$ case follows directly from Theorem \ref{thminterval} since $ M_Ef\leq M_{[1,2]}f$.
\end{proof}

\section{Necessary conditions}\label{sect-nece}
In this section, we prove Theorem \ref{thmnecessary} and the negative results of Theorem \ref{thmhigherdim}. Let $\be= \dim_{\mathrm{M}}\!E, \ga=\dim_{\mathrm{qA}}\!E, \ga_*=\dim_{\mathrm{A}}\!E$. Note that the measure $ \mu_d $ is not translation invariant on $\R^+$. So it seems hard to get the necessary condition $ p\leq q $ directly from the classical result of H\"ormander \cite{1960Hormander}. Fortunately, we can find another simple example to ensure this condition. The counterexamples regarding the other two lines, $ OQ_2(\be)$ and $ Q_1(\be)Q_2(\be)$, are quite standard (see \cite{1997Schlag,1997SS,2021AHRS}). To prove the endpoint necessary conditions on $[Q_1(\be), Q_2(\be)]$ especially when $ \be=1 $, we use a variant of the example in \cite[Section 2.1]{1997SWW}. 

However, there is a  new phenomenon in two dimensions for quasi-Assouad regular sets $E$ when $ 2\ga>\be+1 $. The counterexample for the line segments $ (Q_2(2\ga-1), Q_3(\be,\ga)) $ and $ [O,Q_1(2\ga_*-1)]$ originates from the characteristic function of the annulus, which we have already seen in the proof of Proposition \ref{prop-2DLorentz-necessary}. This example heavily relies on  the singularity of the kernel $ K_t(r,s) $ (see (\ref{kernel})) when $ d=2 $ and indicates the role of the localized covering number $ N(E\cap I,\de)$.
\subsection{The line connecting $O$ and $Q_1(\be)$}\label{stand1}
 Let $ B_R $ be the ball of radius $ R\gg 1 $ centered at the origin in $ \R^d $ and $ \chi_R $ be the characteristic function of $ B_R  $. Choose some $ t_0\in E $. Then we have
\[M_E \chi_R(x)\geq A_{t_0}\chi_R(x)\geq 1,\]
when $ |x|\leq R/2$. Hence, $ R^{d/q}\lesssim R^{d/p} $, which leads to $ p\leq q $ after letting $ R\to\infty $.
\subsection{The line connecting $O$ and $Q_2(\be)$}\label{stand2}
This is the standard dimensional constraint for the fixed time spherical  averages. Choose some $ t_0\in E\subset [1,2]$ and $ f_\de(x)=\chi_{(t_0-\de, t_0+\de)}(|x|)$ with $ 0<\de<1 $. Clearly, $ \|f_\de\|_{L^p_{rad}(\R^d)}\lesssim\de^{1/p}$. Furthermore, $M_E f_\de(x)\geq A_{t_0}f_\de(x)\geq 1$ when $ |x|\leq \de $. Consequently, $ \de
^{d/q}\lesssim \de^{1/p} $, which implies $ q\leq pd $ after letting $ \de\to 0 $.
\subsection{The line connecting $Q_1(\be)$ and $Q_2(\be)$}\label{stand3}
Given $ 0\leq \de<1 $, we take $ f_\de(x)=\chi_{[0,\de]}(|x|)$. Consider the maximal $ \de$-separated set $ \{t_1,t_2,\cdots, t_m\}\subset E $. Obviously, $ m\sim N(E,\de) $. When $ |x|\in A_i:=[t_i-\de/2,t_i+\de/2] $, it is not hard to get that 
\[M_E f_\de(x)\geq A_{t_i}f_\de(x)\gtrsim \de^{d-1}.\]
Note that these annuli are disjoint. Then their union has  Lebesgue measure $\sim \de N(E,\de) $. Consequently, we get
\begin{equation}\label{coveringnumber}
	\de^{d-1}(\de N(E,\de))^{1/q}\lesssim \de^{d/p}.
\end{equation}
For any $\ep>0 $, it follows from the definition of Minkowski dimension that there exists a sequence $ \de_m\to 0 $ such that $ N(E, \de_m)\geq \de_m^{\ep-\be} $. Substituting it into (\ref{coveringnumber}),  letting $ m\to \infty $ first and $ \ep\to 0 $ next,  we obtain the desired necessary condition
\[\frac{1-\be}{q}+d-1\geq \frac{d}{p}.\]
\subsection{Endpoint necessary conditions on $[Q_1(\be), Q_2(\be)]$}
To prove the necessary part of (ii) and (iii) in Theorem \ref{thmhigherdim}, we shall need the following proposition.
\begin{prop}
	Let $ d\geq 2  $ and $ 0\leq \be \leq 1 $. If $ M_E $ is of radial strong type $ (p, q) $ for some $ (1/p,1/q)\in [Q_1(\be), Q_2(\be)] $, then 
	\[\sup_{0<\de<1}\chi^E_{\mathrm{M},\beta}(\delta)<\infty.\]
	Suppose that in addition $ \be=1 $, then 
	\[\sup_{0<\de<1}(\log(\tfrac{1}{\de}))^{\tfrac{q}{d}}\chi^E_{\mathrm{M},1}(\delta)<\infty.\]
\end{prop}
\begin{proof}
	The first assertion follows directly from (\ref{coveringnumber}). For the second necessary condition, it suffices to show that $ |W_n|\lesssim n^{-q/d} $ according to Lemma \ref{lem-binary}.  We may take 
	\[f(x)=|x|^{1-d}\chi_{[2^{-10n},1]}(|x|),\]
	where $ n\gg 1 $. It is clear that $ \|f\|_{L^{p_d}(\R^d)}\sim n^{1/p_d} $.  Here we use the shorthand notation $ p_d=d/(d-1) $ as before. Given $ x\in \R^d $ with $ |x|\in W_n $, there exists a $ t_x\in E $ such that $ ||x|-t_x|\leq 2^{-n+1}$. Set 
	\[A_{j,x}=\{y\in \mathbb{S}^{d-1}\, :\, \big|y-\tfrac{x}{|x|}\big|\in [2^{-j},2^{-j+1}]\}\] 
	for $ 2\leq j\leq n/2 $.  Clearly,  the surface measure of this spherical cap $ A_{j,x} $ is comparable to $ 2^{-j(d-1)} $ and $|x-t_xy|\sim 2^{-j} $ for $ y\in A_{j,x} $. Then, we have, as long as $|x|\in W_n$,
	\[M_E f(x)\geq A_{t_x}f(x)\geq \sum_{j=2}^{n/2}\int_{A_{j,x}}f(x-t_xy)d\sigma(y)\gtrsim n.\] Thus, 
	$$n|W_n|^{1/q}\lesssim \|M_Ef\|_{L^q}\lesssim \|f\|_{L^{p_d}}\sim n^{\tfrac{d-1}{d}},$$
	which implies $ |W_n|\lesssim n^{-q/d}$ as required. 
\end{proof}
\subsection{Proof of Theorem \ref{thmnecessary}}
When $ d=2 $, we have an upper bound for the localized covering number $ N(E\cap I,\de) $.
\begin{prop}
	Let $ d=2 $. Suppose $ M_E $ is of radial restricted weak  type $ (p, q) $. Then, for any interval $ I\subset [1,2] $ and $ \de>0 $ such that $ \de\leq |I|\leq 1$, we have 
	\begin{equation}\label{ineq-crucial}
		\de^{\frac{1}{2}+\frac{1}{q}}|I|^{\frac{1}{q}-\frac{1}{2}}N(E\cap I,\de)^{\frac{1}{q}}\lesssim \de^{\frac{1}{p}}.
	\end{equation}
\end{prop}
\begin{proof}
	For convenience, it suffices to consider the case $I=(u, u+|I|) \subset [1,2]$ with $ \de\leq |I|\leq 1/4 $ and $ 0<\de\ll 1$. We denote the maximal $ \de$-separated set of $ E\cap I $ by $  \{t_1,t_2,\cdots, t_m\}$. It is easy to see that $ N(E\cap I,\de)\sim m $. Without loss of generality, we may assume that at least half of $ \{t_1,t_2,\cdots, t_m\} $ belongs to $ (u+|I|/2, u+|I|) $. For simplicity, we denote them by $ t_1,t_2,\cdots, t_\nu $ with $ \nu\geq m/2 $. Let  $ f_\de(x)=\chi_{[u-\de, u+\de]}(|x|)$ for a give $ 0<\de\ll 1 $. Clearly,  $ \|f_\de\|_{ L_{rad}^{p,1}(\R^2)}\sim\de^{1/p}$. By (\ref{radialformula}) and (\ref{kernel}), we have 
	\[M_E f_\de(x)\geq A_{t_i}f_\de(x)\gtrsim r^{-1/2} \int_{0}^r\frac{1}{\sqrt{s}}\chi_{[u-\de, u+\de]}(s+t_i-r)ds,\]
	when $ r=|x|\leq t_i/2$. For any $ r\in [t_i-u,t_i-u+\tfrac{\de}{2}] $ with $ 1\leq i\leq \nu $, it follows from easy calculation that $ s+t_i-r\in [u-\tfrac{\de}{2},u+\tfrac{\de}{2}] $ whenever $ s\in [0,\de/2] $. Consequently,
	\[M_E f_\de(x)\gtrsim r^{-1/2}\int_{0}^{\de/2}\frac{1}{\sqrt{s}}ds\sim \de^{1/2}|I|^{-1/2}.\]
    if $ |x|\in\bigcup_{i=1}^\nu [t_i-u,t_i-u+\tfrac{\de}{2}] $. These annuli are disjoint since $ t_i $ are mutually $ \de$-separated. Hence their union has Lebesgue measure $ \sim \nu \de |I| $. Note that $ \nu\sim m\sim N(E\cap I,\de)$, which yields the desired bound (\ref{ineq-crucial}).
\end{proof}

\begin{proof}[Proof of Theorem \ref{thmnecessary}]
	By the standard counterexamples from Section \ref{stand2}-\ref{stand3}, $ p\geq \frac{\be+3}{2} $ if $ M_E $ is of  radial restricted weak type $ (p,2p) $. According to the definition of Assouad dimension (\ref{def-assouad}), for any $ \ep>0 $, there exist $ \de_m \to 0$ and $ I_m\subset [1,2] $ of length  $ |I_m|\geq \de_m $ such that
	\begin{equation}\label{sequence1}
		N(E\cap I_m,\de_m)\geq C_m(\de_m/|I_m|)^{-\ga_*+\ep}.
	\end{equation}
	Here $ C_m\to \infty $ as $ m\to\infty $. Set $ A_m=|I_m|/\de_m $. It is easy to see that $ A_m\to \infty $ since there is a trivial estimate $N(E\cap I_m,\de_m)\leq A_m$. Substitute (\ref{sequence1}) into (\ref{ineq-crucial}) for $q=2p$, we have
	\[A_m^{\frac{1+\ga_*-\ep}{2p}-\frac{1}{2}}\lesssim 1.\]
	Letting $m\to \infty $ first and $ \ep\to 0 $ next, we obtain the desired bound $ p\geq 1+\ga_* $. 
	
	Now we consider the line segment with endpoints $Q_2(2\ga-1)$ and $Q_3(\be,\ga)$  for the quasi-Assouad regular set $ E $ when $ 2\ga> \be+1 $. Set $\ga_\theta=\overline \dim_{\mathrm{A},\theta}E $. By the definition of Assouad spectrum, there exist $ \de_m \to 0$ and $ I_m\subset [1,2] $ of length  $ |I_m|\geq \de_m^\theta $ such that
	\[N(E\cap I_m,\de_m)\geq C_m(\de_m/|I_m|)^{-\ga_\theta +\ep},\]
	with $ C_m\to \infty$. Proceeding as before, we have
	\[\de_m^{\tfrac{1}{2}+\tfrac{1-\ga_\theta+\ep}{q}}|I_m|^{\tfrac{1+\ga_\theta-\ep}{q}-\tfrac{1}{2}}\lesssim \de_m^{\tfrac{1}{p}}.\]
	Note that, if $ q<2(1+\ga) $, there exists a small enough $ \ep>0 $ such that $\tfrac{1+\ga_\theta-\ep}{q}-\tfrac{1}{2}>0$  whenever $ \theta>1-\frac{\be}{\ga}$, since $ E $ is quasi-Assouad regular. Then, using  $|I_m|\geq \de_m^\theta $ and letting $ m\to\infty $, we get
	\[\frac{1}{2}+\frac{1-\ga+\ep}{q}+\theta\big(\frac{1+\ga-\ep}{q}-\frac{1}{2}\big)\geq \frac{1}{p},\]
	for $  \theta>1-\frac{\be}{\ga}$ and $ \ep\ll 1 $. Letting $ \epsilon\to 0 $ and $ \theta\to 1-\frac{\be}{\ga} $, we find the necessary condition 
	\[L(\frac{1}{p},\frac{1}{q}):=\frac{1}{2}+\frac{1-\ga}{q}+\big(1-\frac{\be}{\ga}\big)\big(\frac{1+\ga}{q}-\frac{1}{2}\big)-\frac{1}{p}\geq 0.\]
	It follows from a straightforward computation that $ L(Q_1(2\ga-1))=0$ and $ L(Q_3(\be,\ga))=0$.
\end{proof}

\begin{center}
ACKNOWLEDGMENTS
\end{center}

The  author would like to thank Professor Chengbo Wang for his encouragement and helpful discussion. The author would also like to express gratitude to Professor Sanghyuk Lee for the warm hospitality  during his visit to Seoul National University. After completing this paper, the author learned from Professor Seeger that Beltran--Roos--Seeger \cite{2024BRS} has obtained similar high-dimensional results. In two dimensions, they introduced a new  concept in fractal geometry to give a complete characterization of $ \overline{\mathcal{T}_E^{rad}}$ for all given $ E\subset[1,2]$.
\vskip10pt


\end{document}